\documentclass[12pt]{amsart}
\usepackage{amsmath,amssymb,amsthm, verbatim, amscd, amsfonts}
\usepackage{setspace}
\newcommand{\V}{\ensuremath{\vspace{.1in}}}

\newcommand{\R}{\ensuremath{\mathbb{R}}}

\newcommand{\N}{\ensuremath{\mathbb{N}}}
\newcommand{\Z}{\ensuremath{\mathbb{Z}}}
\newcommand{\Q}{\ensuremath{\mathbb{Q}}}
\newcommand{\A}{\ensuremath{\mathbb{A}}}
\newcommand{\T}{\ensuremath{\it{Tri}}}

\theoremstyle{definition}

\theoremstyle{remark}

\theoremstyle{plain}

\newtheorem{theorem}{Theorem}[section]

\newtheorem{lemma}{Lemma}[section]
\newtheorem{cor}{Corollary}[section]
\newtheorem{prop}{Proposition}[section]

\newtheorem{kor}[theorem]{Corollary}
\newtheorem{con}{Conjecture}
\newtheorem{defin}{Definition}

\theoremstyle{definition}

\title{The Density of the Set of Trisectable Angles}
\author{Peter J. Kahn}

\begin{document}

\renewcommand{\thefootnote}{\fnsymbol{footnote}}
\maketitle

\vspace{.2in}

\begin{center}
Department of Mathematics\\
Cornell University\\
Ithaca, New York 14853\\
December 23, 2010\\
Revised July 22, 2011
\end{center}

\vspace{.2in}

\begin{abstract} It has been known for almost 200 years that some angles cannot be trisected by straightedge and compass alone.   This paper studies the set of such angles  as well as its complement $\mathcal{T}$, both regarded as subsets of the unit circle $S^1$.  It is easy to show that both are topologically dense in $S^1$ and that $\mathcal{T}$ is contained in the countable set $\mathcal{A}$ of all angles whose cosines (or, equivalently, sines) are algebraic numbers (Corollary 3.2). Thus, $\mathcal{T}$ is a very ``thin'' subset of $S^1$.  Pushing further in this direction, let $K$ be a real algebraic number field,  and let $\mathcal{T}_K$ denote the set of trisectable angles with cosines in $K$.  We conjecture that the ``computational density'' of $\mathcal{T}_K$ in $K$ is zero and  prove this when $K$ has  degree $\leq 2$ (cf. \S 1.2 and Theorem 4.1). In addition to some introductory field theory, the paper uses elementary counting arguments to generalize a theorem of Lehmer (Theorem 5.2)on the density of the set of relatively prime $n$-tuples of positive integers.
\end{abstract}

\vspace{.2in}

%\tableofcontents

%\doublespacing

\vspace{.2in}
\section{Introduction}\label{S:introduction}

\V

In 1837, P.L. Wantzel proved that there exist angles that cannot be trisected by strict use of straightedge and compass alone \cite{wan} (see also, \S 1.1 below for further discussion).  An easy observation based on his argument (Cor. 3.2) shows that the cosine of each trisectable angle must be an algebraic number, and thus the trisectable angles comprise merely a countable subset of the unit circle $S^1$.  This paper pushes further in the direction of showing how rare trisectable angles are.  To this end, we   let $\mathcal{A}\subset S^1$ be the countable set of all angles with algebraic cosines and let $\mathcal{T}\subseteq\mathcal{A}$ be the subset of trisectables. The main result of this paper (Theorem 1.1 or Theorem 4.1) provides evidence for the conjecture that $\mathcal{T}$ is a very ``thin'' subset of $\mathcal{A}$.

\V

Statements of our results will be given at the end of this introduction, as will a description of how the remainder of the paper is organized.  First, however, we present some  context for the angle trisection problem.

\subsection{A thumbnail history of the angle trisection problem}

\vspace{.1in}

The classical angle trisection problem, which originated with Greek mathematicians in the 5th century B.C.E.,  requires that for each angle $\alpha$ one find a geometric procedure that starts with  $\alpha$ and ends with $\alpha/3$.  From that period on, numerous solutions to the problem have been given \cite{dud}, many of them very ingenious. None of these, however, make exclusive use of the simple straightedge and compass, requiring instead auxiliary devices: e.g., marks on the straightedge, special curves, such as the Quadratrix of Hippias and the Conchoid of Nicomedes \cite{dud}, and other devices, such as the ``Shoemaker's knife''\cite{hen}.

A purist strain in Greek geometry, said to originate with Plato and encoded in the axiomatics of Euclid's \emph{Elements}, devalued the use of such auxiliary methods in geometric constructions. Hence constructions using only compass and straightedge  have been called Euclidean \cite{dud}.
What has come to be called ``the angle trisection problem'' in subsequent times is the problem of finding, for each $\alpha$, a \emph{Euclidean} geometric construction that starts with  $\alpha$ and produces $\alpha/3$.

A \emph{Euclidean geometric construction} may be described more precisely as a finite sequence of steps starting with at least two distinct points in the plane  such that each step is a construction of one of the following two types:  (1) If two distinct points exist at a prior step, then the line they determine or the circle centered at one of them and passing through the other may be constructed.  (2) If two lines, two circles, or a line and a circle exist at a prior step, then their points of intersection, if any, may be constructed.

\emph{In this paper, all further references to constructions, trisections, trisectability, and the like, will assume the Euclidean restriction.}

After many failed efforts at solving the Euclidean angle trisection problem, it became widely believed, even in Euclid's time, that a solution was impossible. However, this belief was not supported by proof.

Further progress on the Euclidean angle  trisection problem was not achieved until the development of trigonometry and algebra by Arab mathematicians.   Around 1430, the Arab mathematician Al-K\=ash\=i showed that the  trisectability of a given angle $\alpha$  is related to the solvability of a certain cubic polynomial \cite{hog}.  Presumably, this is essentially the polynomial $q(x,b)$ described in the following paragraph, which derives from the trigonometric identity  expressing $\sin(\alpha)$ in terms of $\sin(\alpha/3)$.   In 1569, the Italian mathematician Rafael Bombelli independently demonstrated  the same connection between trisectability and algebraic solvability \cite{bor}.  However, with both Al K\=ash\=i and  Bombelli, the concepts of ``trisectability'' and ``solvability'' remain unclear.  Neither mathematician gives an  algebraic criterion for Euclidean constructibility.  Nor does either specify what kind of algebraic solutions  he has in mind.  Finally, neither makes a claim to having proved the impossibility of solving the trisection problem.   Nevertheless, despite these gaps, the connection these mathematicians obtain between the purely geometric problem and an algebraic one was a major breakthrough that foreshadowed the ultimate solution

This solution was finally and decisively achieved by the French mathematician Pierre-Laurent Wantzel (1814-1848)   \cite{wan}, as stated at the outset of this paper. By that time, the connection between trisectability and properties of the cubic
polynomial $q(x,b)=4x^3-3x+b$  must have been well known, because Wantzel uses the polynomial without further comment.  The parameter $b$ represents the sine of a given angle $\alpha$, and one of the roots of $q(x,b)$ is $\sin(\alpha/3)$.   Wantzel demonstrates that constructible quantities must be zeros of irreducible polynomials of degree a power of two, and he  uses this criterion to deduce that   $\alpha$ is trisectable if and only if  $q(x,b)$ is reducible over the field $\Q(b)$.  (To be sure, he does not use the language of fields, since the concept of a field was not fully developed until the 1850's.) Since there are many $b$ in the interval $[-1,1]$ for which $q(x,b)$ is \emph{not} reducible over $\Q(b)$, there are, correspondingly,  many angles for which there is no  Euclidean trisection procedure (cf. \textbf{Corollary 2.2}).

\V

\subsection{ Some terminology, statements of results and organization of the paper}

\V

For technical convenience in our work and in the statements of results, we shall replace $\sin(\alpha)$ by $2\cos(\alpha)$.  This amounts to replacing the polynomial $q(x,b)$ by the polynomial $p(x,a)= x^3-3x-a$,  which we use throughout the rest of this paper. Here the parameter $a$ represents $2\cos(\alpha)$.  Wantzel's argument applies equally well to $p(x,a)$, so we can rephrase his result as:  \emph{$\alpha$ is trisectable if and only if $p(x,a)$ is reducible over the field $\Q(a)$.}\V

Next, we normalize our discussion by using the Cartesian plane $\R^2$, with its usual terminology and notation.
Points in $\R^2$ will be called \emph{constructible} if they can be obtained from the  set of points $\{(0,0), (1,0)\}$ by a Euclidean construction.    More generally, given any set  $S\subseteq \R^2$ that contains  $\{(0,0), (1,0)\}$, any point for which there is a construction starting with points in $S$ is said to be \emph{constructible over $S$}.  The set of all such points will be denoted by $C(S)$.   An angle $\alpha$  is identified in the usual way with the point $(\cos(\alpha), \sin(\alpha))$ on the unit circle $S^1$, which allows us to talk about \emph{constructible angles}.     $\alpha$ is said to be \emph{trisectable} if $\alpha/3\in C(\{(0,0), (1,0), \alpha\})$ \footnote{The notation ``$\alpha/3$'' here represents $1/3$ of $\alpha$ in angular measure and should not be confused with scalar multiplication of the point $\alpha$ by the scalar $1/3$. A similar caveat applies to the examples of angles that we give later. }. Since this last relation bears no obvious connection to the relation $\alpha\in C(\{(0,0), (1,0)\})$,  there is no reason to suppose that a trisectable angle need be constructible. \V

Some examples are in order.  First note that the following angles are  \emph{constructible}: $\alpha=\pi/3;\;\beta_n= \pi/2^n;\; \gamma_n=\pi/3\cdot2^n;\; \epsilon_n= \pi/2^n + \pi/3$. The reasons are: $\alpha$ is the angle in an equilateral triangle; $\beta_n$ can be obtained from $\pi$, which is obviously constructible,  by repeated bisection; $\gamma_n$ can be obtained from $\alpha$ by repeated bisection;   $\epsilon_n=\beta_n+\alpha$.  Here, $n$ is any non-negative integer.\V

Now, $\alpha=\pi/3$ is the most commonly presented example of a \emph{non-trisectable} angle: for  $2\cos(\pi/3)=1$, and $p(x,1)$ is easily shown to be irreducible. Therefore, we may conclude that $\epsilon_n$ is not trisectable: for if it were, then starting with $\epsilon_n$ we could construct $\epsilon_n/3 =\gamma_n+ \alpha/3 $, and from that we could subtract $\gamma_n$, obtaining $\alpha/3.$ Since $\epsilon_n$ is constructible, we could conclude that $\alpha/3$ is constructible, hence, \emph{a fortiori} constructible over $\{(0,0), (1,0), \alpha\}$, which was just shown to be impossible.\V

By the same arguments, we could conclude that all angles of the form $k\beta_n$ are trisectable, $k$ an arbitrary integer,  and all angles of the form $k\beta_n+ \alpha$ are non-trisectable.  Both these sets are countable, dense subsets of the unit circle. Of course, all of these examples are constructible. \V

At this point it becomes convenient to focus our attention away from the angles themselves and toward their cosines.  Consider any number $a\in[-2,2]$.  We call $a$ a \emph{trisection number} if there is a trisectable angle $\alpha$ such that $a=2\cos(\alpha)$, and we denote the set of all trisection numbers by $Tri$.  As we show in Corollary 3.2 (and have already mentioned earlier), $Tri\subseteq\A$, where $\A$ is the set of algebraic numbers. For any real algebraic number field $K$, we define the \emph{density} of $Tri\cap K$ in $[-2,2]\cap K$ in \S 4, and we denote it by $\delta_K(Tri)$.

\V

\begin{con}\label{C:mainconjecture} $\delta_K(Tri)=0.$
\end{con}

\begin{theorem}[Main result] Conjecture \ref{C:mainconjecture} is true when $K$ is the field of rational numbers \Q\ or a real quadratic field.
\end{theorem}

\V

Now we present a selection of other results in the paper which may be of independent interest.

\vspace{.1in}

The following three results in \S\S 3.2, 3.3 give further examples of non-trisectable angles :

\vspace{.1in}

\noindent\textbf{Corollary 3.2}\emph{ $\alpha$ is non-trisectable  whenever $\cos(\alpha)$ is transcendental. Therefore, \emph{Tri} is countable and the set of non-trisectable angles is uncountable.}

\vspace{.2in}

\noindent\textbf{Proposition 3.2} \emph{If $a$ is a non-zero  square in \Q, then $a\not\in \T$.}

\vspace{.1in}

\noindent\textbf{Proposition 3.3} \emph{ Let $r$ and $s$ be any non-zero integers prime to each other and to $3$. Then $p(x,3r/s)$ is irreducible over \Q.  Hence,  $3r/s\not\in \T$ .}

\vspace{.1in}

Lest the reader be left wondering whether there are any non-trisectable angles with irrational algebraic cosines, we prove the following result in Appendix C  about the non-trisectable angles $\pi/3 + \pi/2^n$:

\vspace{.1in}

\noindent\textbf{Theorem}\quad The  numbers $\cos(\pi/3 + \pi/2^n)$ are algebraic of degree $2^n$.

\vspace{.1in}

\noindent\textbf{Proposition 3.4}, together with its addendum, implies that \emph{there exist a countable number of trisectable angles that are not constructible}.

\vspace{.1in}

\noindent\textbf{Theorem 4.1}\quad \emph{ Let $K$ be a real number field of degree $k\leq 2$.  Then}
\[ \delta_K(R)\quad is\quad \mathcal{O}(R^{-\frac{2}{3}(k+1)}).\]

\vspace{.1in}

Here $\delta_K(R)$ counts the number of trisection numbers in $K$ that have height $\leq R$ and divides this by the number of members of $[-2,2]\cap K$ of height $\leq R$.  See \S 4
for the definition of height and further notation.  Using the usual definition of the big ``$\mathcal{O}\,$'' notation, it is clear that Theorem 4.1 implies Theorem 1.1.

\vspace{.1in}

In \S\ 5,  we estimate the number of relatively prime positive integer $k$-tuples $(a_1,\ldots,a_k)$ satisfying $a_i\leq n_i$, where $\mathbf{n}=(n_1,\ldots,n_k)$ is a $k$-tuple of positive real numbers
(\textbf{Theorem 5.2)}.\quad This generalizes a theorem of Lehmer and Sittinger \cite{sit}. A consequence of this theorem is an asymptotic relation, which is easier to state than the theorem itself:
\[|C(k,\mathbf{n})| \sim \frac{n_1\cdot\ldots\cdot n_k}{\zeta(k)}.\]
The term on the left denotes the number of integer $k$-tuples $(a_1,\ldots,a_k)$ being counted and $\zeta$ denotes the Riemann zeta function.  See \S\ 5 for further definitions.

\vspace{.1in}

Finally, in Appendix A, we prove some results about $n$-section of angles, by which we mean a Euclidean construction that starts with an angle $\alpha$ and produces the angle $\alpha/n$.

\vspace{.1in}

\noindent\textbf{Theorem A.1}\quad \emph{Suppose $n$ is a positive integer such that for any given angle $\alpha$,  there is a Euclidean  construction that starts with  $\alpha$ and ends with $\alpha/n$.  Then, $n$ has the form $2^k$, for some non-negative integer $k$.}\vspace{.1in}

Therefore, for any $n$ not a power of $2$, there exist non-$n$-sectable angles.  But can we assert, as in the case of non-trisectable angles, that, for fixed $n$ (not a power of $2$), the set of all non-$n$-sectable angles is dense in $S^1$, or similarly for the set of all $n$-sectable angles? \vspace{.1in}

 In general, this may not be so easy.  However, we can assert this to be the case for special $n$: \V

\noindent\textbf{Propositions A.1 and A.2:}\quad \emph{If $n$ is an odd positive integer such that $2\pi/n$ can be constructed (i.e., the regular polygon of $n$ sides can be constructed), then there exists a countable dense subset of $S^1$ consisting of $n$-sectable angles and a countable dense subset of $S^1$ consisting of non-$n$-sectable angles.} \vspace{.1in}

As is well known, Gauss showed that $2\pi/n$ can be constructed provided $\phi(n)$ is a power of $2$.  Here, $\phi$ denotes the Euler phi function.  Among odd integers $n$ for which Gauss's condition holds are the integers $n=3, 5, 17, 257, 65537.$\vspace{.1in}

All of the angles in these two propositions have cosines that are algebraic numbers.  The transcendental case is covered by the following theorem in Appendix A, which extends Corollaries \ref{C:transcendentalimpliesirreducible} and \ref{C:cardinalityoftrisectibleangles} of the main text.

\vspace{.2in}

\noindent\textbf{Theorem A.2}\quad\emph{ Suppose that $cos(\alpha)$ is transcendental and that $n$ is a positive integer that is not a power of $2$.  Then $\alpha$ is not $n$-sectable.  Therefore, the set of all non-$n$-sectable angles is uncountable, and the set of all $n$-sectable angles is countable.}

\vspace{.3in}

The paper has three main parts. First,  Sections 1 - 3 consist of background and several results that produce classes of examples of both trisectable and non-trisectable angles.  Second, Sections 4 - 7 prove our main result on density (the precise version of which is Theorem 4.1).  The  proof involves a variety of counting and estimation arguments, including a generalization of a theorem of Lehmer.  Third, the Appendix contains miscellaneous supplementary results. One section (Appendix B) shows that our main result is independent of choice of basis. Another section (Appendix A)  extends some of the results of Sections 2 and 3 to the case of $n$-section of angles. And,  finally, Appendix C shows that the algebraic number  $\cos(\pi/3+\pi/2^n)$ has degree $2^n$.

\vspace{.2in}

At this point, I wish to thank Ravi Ramakrishna for a number of  helpful conversations and suggestions.  I also wish to thank George Wilson and Michael Nussbaum for their help with the Italian article on Bombelli \cite{bor}.   Michael Nussbaum's assistance was particularly helpful in enabling me to assess  Bombelli's contribution to the angle trisection problem. Finally, I wish to thank Benjamin Kahn and Kay Wagner for some interesting questions.  These are answered in Theorems A.1 and A.2.

\vspace{.2in}

\tableofcontents

\section{Basic facts about constructibility and trisectability}\label{S:preliminaries}

Most of the results described in subsections 2.1 and 2.2 are either well known or easily derivable.  They are presented here as background for the reader.

\subsection{Constructible points and numbers} The basic facts about constructible points and numbers are carefully described in various classical texts (e.g., \cite{wae}, \cite{cou}), so we give only a brief outline here.

Let $S$ be a subset of $\R^2$.  We have already defined the set of points $C(S)$ constructible over $S$ in \S 1.2.  We shall now identify  in the usual way the field \R\ of real numbers with the set of all points in $\R^2$ of the form $(x,0)$.  If $S\subseteq \R$, then we call the elements of $C_{\R}(S)=C(S)\cap\R$ \emph{numbers constructible over S} (or simply \emph{constructible numbers} when $S=\{(1,0), (0,1)\}$).

Because the four elementary operations of arithmetic can be realized by Euclidean constructions, it follows that the sets $C(S)$ are closed under these operations, so these sets are subfields of \R.
Furthermore, it follows directly from definitions  that $C_{\R}(C_{\R}(S))=C_{\R}(S)$.  Thus, we lose no generality by assuming that $S$ itself is already a subfield of \R. The further equality $C_{\R}(S)\times C_{\R}(S)=C(S)$  shows that we lose no information about points constructible over $S$ by focusing on numbers constructible over $S$.

The field $C_{\R}(\Q)$ is called the field of \emph{constructible numbers}, and its subfields are called  \emph{constructible fields}.  It follows from our comments above that if $K$ is a constructible field, then $C_{\R}(\Q)=C_{\R}(K)$.

 Now consider some Euclidean construction over $S$.
 By elementary plane geometry, the coordinates of each newly constructed point are zeros of  polynomial  equations of degree at most two, and  the coefficients in each such equation  are rational functions of the coordinates of points  that have already been constructed over $S$. This description leads immediately to the following  fundamental algebraic fact about constructible numbers:

\begin{theorem}\label{T:fundamentalcriterionforconstructibility}  Let $K$ be a subfield of \R.  A real number $x$ belongs to $C_{\R}(K)$ (i.e., is constructible over $K$) if and only if there is a finite tower of real, quadratic field extensions
\[ K=K_0\subset K_1\subset\ldots\subset K_n,\]
such that $x\in K_n$.
\end{theorem}

It follows immediately that $C_{\R}(\Q)$ is a real subfield of the field \A\ of algebraic numbers.

Suppose that the real number $b$ is constructible over the real field $K$ and that $K=K_0\subset\ldots\subset K_n$ is a tower as above with $b\in K_n$.  Then, by Theorem \ref{T:fundamentalcriterionforconstructibility}, the minimal polynomial of $b$ over $K$   must have degree of the form $2^k$ for some $k\leq n$.  When $K=\Q$, this is called the \emph{degree of $b$}.  Therefore, every $b\in C_{\R}(\Q)$ has degree a power of 2.

\subsection{Trisectable angles}

We recall that the angle $\alpha$ is \emph{trisectable} if the angle
$\alpha/3$ is constructible  over the set $ \{(0,0), (1,0),\alpha\}$

We set $a=2\cos(\alpha).$ It is easy to see that $\alpha/3$ is constructible over $ \{(0,0), (1,0),\alpha\}$ if and only if  $\cos(\alpha/3)$ is constructible over the field $\Q(\cos(\alpha))$ or,  equivalently,  $2\cos(\alpha/3)$ is constructible over the field $\Q(2\cos(\alpha))=\Q(a)$. This second formulation is often slightly more convenient for our algebraic computations.  We use either formulation without further comment.

It is possible for the angle $\alpha$ to be constructible without being trisectable (e.g., $\pi/3$, as mentioned in the introduction) and to be trisectable without being constructible.  We give examples of the latter in \S 3.4.

We now invoke a standard trigonometric identity to relate the quantities
$2\cos(\alpha)$ and $2\cos(\alpha/3)$:
\begin{equation}\label{E:basicequation}
 2\cos(\alpha) = (2\cos(\alpha/3))^3-3(2\cos(\alpha/3)).
 \end{equation}
That is, using $a=2\cos(\alpha)$, as above, $2\cos(\alpha/3)$ is a zero of the monic polynomial \[p(x,a)=x^3-3x-a\in \Q(a)[x].\]

\begin{theorem}\label{T:citerionfortrisectability} The angle $\alpha$ is trisectable if and only if $p(x,a)$ is reducible over the field $\Q(a)$, where $a=2\cos(\alpha)$.
\end{theorem}

As indicated in our introduction, an equivalent fact was demonstrated by Wantzel.  We provide a modern version of his proof here for the reader's convenience.

\begin{proof}
$\Leftarrow$: Assume that $2\cos(\alpha/3)$ is constructible over $\Q(\cos(\alpha))=\Q(a)$.  Then , by Theorem 1,  the minimal polynomial $f$ of $2\cos(\alpha/3)$ over $\Q(a)$ has degree $2^n$ , for some integer $n$. Thus, the degree of $f$ is not equal to $0$ or $3$.  Since $f$ divides $p(x,a)$, $p(x,a)$ is reducible in $\Q(a)[x]$.

$\Rightarrow$:  If $p(x,a)$ is reducible over $\Q(a)$, it factors as the product of a linear term and a quadratic term in $\Q(a)[x]$.  Since $2\cos(\alpha/3)$ is a zero of $p(x,a)$,  by the foregoing trig identity, it must be a zero of one of the factors.  Therefore,  in either case, it is constructible over $\Q(a)$, by Theorem 1,  and so $\alpha$ is trisectable.
\end{proof}

\vspace{.2in}

\begin{kor} $\pi/3$ is not trisectable, which means that it is impossible to find a Euclidean trisection construction for each angle.
\end{kor}
\begin{proof}   When $\alpha=\pi/3$, we have $2\cos(\alpha)=1$, so that $p(x,a)=x^3-3x-1$.  The reader can easily check by a direct computational argument that this polynomial is irreducible over $\Q(1)=\Q$.  Alternatively, we give an argument that uses Eisenstein's criterion \cite{wae}.  The polynomial $f(x)=
p(x-1,1) = x^3+3x^2-3$  satisfies the conditions of Eisenstein's Theorem.  Therefore, $f(x)$ is irreducible. It follows that $p(x,1)=f(x+1)$ is irreducible.  Now apply the preceding theorem.
\end{proof}

\vspace{.2in}

\section{Examples of trisectable and non-trisectable angles}

\subsection{Multiples of $\pi$}

As we already indicated in the introduction,  every integral multiple of $\pi/2^n$ is \emph{both} trisectable \emph{and} constructible and  these  form a countable, dense subset of $S^1$.  Indeed, it is easy to see that the set of all angles that are both trisectable and constructible form a dense \emph{subgroup} of $S^1$ under angle addition.  It follows from Corollary \ref{C:cardinalityoftrisectibleangles} that this subgroup is countable.  We do not make use of the group structure in this paper.

\vspace{.2in}

R. C. Yates \cite{yat} proves the following easy generalization of the fact that the integer multiples of $\pi/2^n$ are trisectable.

\begin{prop} Suppose that the integer $k$ is not a multiple of $3$.  Then every multiple of $\pi/k$ is trisectable.
\end{prop}
\begin{proof}  Since $k$ is relatively prime to $3$, there are integers $a$ and $b$ such that $3a+bk=1$.
Multiply both sides of this equation by $\pi/3k$: $a(\pi/k)+b\pi/3= \pi/3k$.  Since $\pi/k$ is given and $\pi/3$ is constructible, it follows that $\pi/3k$ can be constructed from $\pi/k$.\footnote{Yates gives a faulty proof of the inverse of this proposition.  He succeeds only in proving that if $k$ is a multiple of $3$ \emph{and} if  $\pi/k$ is constructible, then it is not trisectable} \end{proof}

Suppose now that $\alpha$ is a constructible angle that is trisectable and $\beta$ is a non-trisectable angle.  Then $\alpha+\beta$ cannot be trisectable.  For if it were, then, starting with $\beta$, we could construct $\alpha +\beta$, hence $(\alpha+\beta)/3$, and from that, $\beta/3$, contradicting the non-trisectability of $\beta$.

Since we have shown above  that there is a countable dense subset of $S^1$ consisting of angles that are constructible and trisectable, and since we have seen that non-trisectable angles exist, this shows that \emph{the set of non-trisectable angles contains a countable subset that is dense in $S^1$}.  Shortly, we demonstrate a much stronger result.

\subsection{Trisection numbers and the countability of the set of trisectable angles}

We give further examples of  non-trisectable angles.  First, we restate the definition of a \emph{trisection number} for emphasis.

\begin{defin} A \emph{trisection number} is any real number of the form $2\cos(\alpha)$, for some trisectable angle $\alpha$.
We denote the set of all trisection numbers $\emph{Tri}$.
\end{defin}

\vspace{.1in}

We begin with the following lemma.

\begin{lemma} \label{L:indeterminateparameter} Let $t$ be an indeterminate. The polynomial $p(x,t)\in\Q(t)[x]$
is irreducible.
\end{lemma}

\begin{proof}  Suppose the conclusion is false.  Then $p(x,t)$ has a zero in $\Q(t)$, which we may write as $A/B$, where $A$ and $B$ are relatively prime polynomials in $\Q[t]$.  The equation $p(A/B,t)=0$ implies that $A(A^2-3B^2)=tB^3$,
and this, in turn implies that $B|A^2$, which is impossible unless $B$ is a unit in $\Q[t]$, i.e., a non-zero constant polynomial, which we may absorb in $A$.  Therefore, $p(A,t)=0$, which implies $A\neq 0$ and $A|t$.  So, either there is a non-zero rational number $c$ such that  $A=c$ or there is a non-zero $c$ such that $A=ct$.  Either choice gives a non-trivial algebraic relation satisfied by $t$ over \Q\,, $p(A,t)=0$, which is impossible.
\end{proof}

See the proof of Theorem A.2 in Appendix $A$ for an alternative proof of this lemma that uses Eisenstein's criterion.

\vspace{.05in}

Since $\Q(a)\equiv \Q(t)$ when $a$ is transcendental, the following corollary is immediate:

\begin{cor}\label{C:transcendentalimpliesirreducible} $p(x,a)$ is irreducible over $\Q(a)$, for every transcendental $a$.
\end{cor}

\begin{cor}\label{C:cardinalityoftrisectibleangles} $\alpha$ is non-trisectable  whenever $\cos(\alpha)$ is transcendental. Therefore, \emph{Tri} is countable and the number of non-trisectable angles is uncountable.
\end{cor}
\begin{proof}  Let $a=2\cos(\alpha)$ and $b=2\cos(\alpha/3)$.  We have shown that the following assertions are equivalent: $\alpha$ is non-trisectable; $b$ is not constructible over $\Q(a)$; $p(x,a)$ is irreducible over $\Q(a)$.  Thus, by  Corollary \ref{C:transcendentalimpliesirreducible}, when $a$   is transcendental, $\alpha$ is not trisectable.  Therefore, $\T$ is a subset of $\A\cap\R$, and so it is countable. Since $\alpha\mapsto a$ is at most $2-1$ and is onto $[-2,2]$, the set of non-trisectable $\alpha$ is uncountable.
\end{proof}

\vspace{.1in}

This corollary shows that to find a trisectable angle $\alpha$ (or rather, to find the corresponding number $2\cos(\alpha)=a$), we  can  require, without loss of generality,  that we search for $a$  among the \emph{algebraic numbers. Therefore, we assume throughout the rest of this paper that $a$ is algebraic.}

\vspace{.1in}

In Appendix A, we generalize Corollary \ref{C:cardinalityoftrisectibleangles} to the case of  $n$-sectable and non-$n$-sectable angles, for every $n$ that is not a power of $2$.

\subsection{Non-trisectable angles with rational cosines}

We now focus on the numbers $a=2\cos(\alpha)$ to describe some examples of non-trisectable angles.   The next two results display countably many examples of \emph{rational} $a\in [-2,2]$ that are not trisection numbers.  Contrast these with the non-trisectable angles $\frac{\pi}{3}+\frac{\pi}{2^n}$, whose cosines are constructible numbers of arbitrarily high degree over \Q\ (cf. Appendix C).

\vspace{.1in}

\begin{prop} If $a$ is a non-zero  square in $\Q\cap[-2,2]$\, then $a\not\in \emph{Tri}$.
\end{prop}
\begin{proof}  Let $E$ be the projective elliptic curve whose affine equation is
$y^2=x^3-3x$.  It is well known that the set of rational points on $E$ form a finitely-generated abelian group of rank $0$ and torsion group $\mathbb{Z}_2$
(e.g., see \cite{hus}, pp.33-35).  That is, the only rational points on $E$ are $[0,0,1]$
 and $[0,1,0]$, the point at infinity.  It follows that there is no non-zero rational $c$ such that $c^2=x^3-3x$ has a rational solution. So, $p(x,c^2)$
is irreducible over $\Q=\Q(c^2)$, for all non-zero, rational $c$.  The result now follows from Theorem 2
by restricting to non-zero, rational $c$ such that $c^2\leq 2$.
\end{proof}

\begin{prop} Let $r$ and $s$ be any non-zero integers prime to each other and to $3$. Then $p(x,3r/s)$ is irreducible over \Q.  Hence,  no such $3r/s$ in $[-2,2]$ belongs to $Tri$.

\end{prop}
\begin{proof}  The irreducibility of $p(x,3r/s)$ is an immediate consequence of the Eisenstein Criterion
and the Gauss Lemma.
\end{proof}

\noindent\textbf{Remarks:}\quad a)  Any real number field $K$ (i.e., subfield of \R\ of finite degree over \Q) whose integers admit unique factorization can be used in place of \Q\ in this proposition, provided $3$ is a prime in the ring of integers of $K$.

b) The foregoing results show that both $K\cap\T$ and $(K\cap[-2,2]) \setminus (K\cap\T)$ are big subsets of $K\cap[-2,2]$ for many real number fields $K$.  In the next few sections, we show  in a number of cases that  $K\cap\T$ is much the smaller of the two.

\subsection{Non-constructible trisection numbers}

We conclude with a family of examples of real number fields containing countably many trisection numbers that are not constructible.  We remind the reader that these correspond to trisectable angles that are not constructible.

First, it will be convenient to make use of the polynomial function $f:\R\rightarrow \R$ given by the equation $f(x)= x^3-3x$, so that, for each real $x$ and $a$, $p(x,a)=f(x)-a$.  Set $y=f(x)$. It is easy to check that if $-2\leq x \leq 2$ (resp, $x <-2,\; x>2$), then $-2\leq y \leq 2$ (resp., $y <-2,\; y>2$).  It follows immediately  that $f([-2,2])=[-2,2]$ (resp., $f^{-1}([-2,2])=[-2,2]$ , so that $f(S\cap[-2,2])=f(S)\cap[-2,2]$ (resp., $f^{-1}(S\cap[-2,2])=f^{-1}(S)\cap[-2,2]$) for every $S$. The following lemma will be useful for our later computations.  We leave the easy verification to the reader.

\begin{lemma}  Let $K$ be any subfield of\;  \R.  Then $\T\cap K\subseteq f(K\cap[-2,2])$.
\end{lemma}

The following proposition now forms the basis for the examples just mentioned.

\begin{prop}\label{P:nonconstructibletrisectionnumbers}  Let $K$ be a finite, constructible extension of \Q,  and let $F$ be a real, finite  extension of $K$ such that $[F:K]$ is   prime to $2$ and $3$.  Then:  (a) $\T\cap F= f(F\cap[-2,2])$; \quad (b)  No element of $f(F\setminus K)$ is constructible.  In particular, when $F\neq K$, $f(F\setminus K)$ contains a countable infinity of trisection numbers that are not constructible.
\end{prop}

\begin{proof} (a) By the lemma, it suffices to prove that $\T\cap F\supseteq f(F\cap[-2,2])$.  So, choose any $a\in f(F\cap[-2,2])=f(F)\cap[-2,2]$.  That is, $a$ is in $[-2,2]$ and is of the form $a=b^3-3b$, for some $b\in F$. The degree $d$ of $b$ over $K(a)$ divides
$[K(b):K]$, which, in turn, divides the odd number $[F:K]$.  Since $b$ is a zero of $p(x,a)\in K(a)[x]$, $d$ cannot be greater than $3$.  Since it is odd and prime to $3$, it must equal $1$, so $K(a)=K(b)$. By construction, $p(x,a)$ is reducible over $K(b)$, so it is reducible over $K(a)$.  Now, since $a\in [-2,2]$, we may write $a$ as $a=2\cos(\alpha)$, for some angle $\alpha$.  Then $2\cos(\alpha/3)$ is a zero of $p(x,a)$ by the definition of $p(x,a)$.  We now argue as in the proof of Theorem \ref{T:citerionfortrisectability}. Since $p(x,a)$ factors into the product of a linear and a quadratic polynomial over $K(a)$, and since $2\cos(\alpha/3)$ is a zero of one of these, Theorem \ref{T:fundamentalcriterionforconstructibility} implies that $2\cos(\alpha/3)$ is constructible over $K(a)$.

Now let $\Q=L_0<L_1<\ldots< L_n=K$ be a tower of real quadratic extensions, which exists by the hypothesis on  $K$.  Then, $\Q(a)=L_0(a)\leq\ldots\leq L_n(a)=K(a)$ is a tower of real field extensions, each at most quadratic.  It follows that $2\cos(\alpha/3)$ is constructible over $\Q(a)$, which shows that $\alpha$ is trisectable.

Therefore, $a=2\cos(\alpha)\in \T\cap F$, as required.

(b)\quad Now suppose that $b\in F\setminus K$, and set $a=f(b)$, as above.
The argument in (a) shows that $K(a)=K(b)$, so, in particular $a\not\in K$. Therefore
$[K(a):K]\neq 1$.  Further,  $[K(a):K]$  is odd, since it divides $[F:K]$.  Therefore, by Theorem \ref{T:fundamentalcriterionforconstructibility}, $a$ is not constructible over $K$. So it is not constructible over \Q. Since $F\cap[-2,2]\setminus K$ is infinite, and $f$ is at most three to one,\; $f(F\cap[-2,2]\setminus K)$ is an infinite set of trisection numbers (by part (a)) none of which is constructible.

\end{proof}

\noindent\textbf{Addendum to Proposition \ref{P:nonconstructibletrisectionnumbers}:}\quad \emph{Let  $K$ be any real number field and $m$ any positive integer. There  exists an extension $F$ of $K$ such that $[F:K]=m$.  When $m$ is odd, we may choose the extension to be real.}

Therefore, there exist many instances of constructible fields $K$ and extensions $F$ of $K$ as described by Proposition \ref{P:nonconstructibletrisectionnumbers} , i.e., the proposition is not vacuous.

\vspace{.1in}

We thank Ravi Ramakrishna for suggesting the following proof of the addendum, which we give in three steps.

\vspace{.1in}

\noindent\textbf{Step 1:}\quad   Let $R$  be the ring of integers of $K$, and choose any proper prime ideal $\mathfrak{p}\subset R$. Let $S$ be the localization $R_{\mathfrak{p}}$ of $R$ at $\mathfrak{p}$, and let  $\mathfrak{q}$ be the extension of $\mathfrak{p}$ to $S$.

\vspace{.1in}

\noindent\textbf{Step 2:}\quad  $\mathfrak{q}$ is the unique maximal ideal of the local ring $S$. The non-zero ideals of $S$ are precisely the non-negative powers of  $\mathfrak{q}$, all of which are distinct.  Choose any $q\in \mathfrak{q}\setminus\mathfrak{q}^2$. Since $S$ is a Dedekind domain, the ideal $(q)$ equals a unique non-negative power of $\mathfrak{q}$, which must be the first power.  Therefore, $S$ is a principal ideal domain.

\vspace{.1in}

\noindent\textbf{Step 3:}\quad By Step 2, we may apply  Eisenstein's criterion to the polynomial $x^m-q\in S[x]$, concluding that it is irreducible. Since $K$ is the field of fractions of $S$, Gauss's Lemma implies that $x^m-q$ is irreducible in $K[x]$.
Letting $c$ be any zero of $x^m-q$ (real, if $m$ is odd), the field $F=K(c)$ satisfies the desired condition.

\vspace{.2in}

\section{ The density of $K\cap\T$ in $K\cap[-2,2]$: preliminaries and an overview}\label{S:densityoverview}

We have seen  that the set \T\ of trisection numbers consists of real algebraic numbers
in $[-2,2]$, i.e.,$ \T\subseteq \A\cap\R\cap[-2,2]$.  As a step toward getting more information about the global structure
of \T, we specialize to a  number field $K\subset \A\cap\R$, and we attempt to compute the density of $\T\cap K$ in $K\cap[-2,2]$.

\subsection{Height and density}\label{SS:heightanddensity}

One way to define density in this context is to make use of a so-called \emph{height function}
\[ h_K:K\rightarrow (0,\infty).\]
The definition of $h_K$ that we have in mind is a simplified version of what is used in Diophantine Geometry (cf. \cite{lan}). We begin by choosing a fixed  \Q-vector space basis $\mathcal{V}=
\{v_1,v_2,\ldots v_k\}$ of $K$.

\begin{lemma} Every $x\in K$ can be written uniquely as
\begin{equation}\label{E:elementsofK}
x=(a_1v_1+\ldots+a_kv_k)/b,
\end{equation}
for integers $a_1,\ldots,a_k,b$ satisfying \begin{enumerate}
 \item $b>0$ and
 \item $a_1,\ldots a_k,b$ have no prime factors in common.
 \end{enumerate}
\end{lemma}
We leave the proof to the reader.

Then, using (\ref{E:elementsofK}), we define
\begin{equation}
h_K(x)=\max\{|a_1|,\ldots,|a_k|, b\}.
\end{equation}

For any real, positive $R$, the set
\[B_K(R)\overset{def}{\equiv}h^{-1}_K(0,R]\]
 is finite, and so,
its cardinality $|B_K(R)|$ is a non-negative integer.  Clearly, if $\mathcal{R}$ is any unbounded subset of $(0,\infty)$, then
\[ \bigcup_{R\in\mathcal{R}}B_K(R) = K.\]

For sufficiently large $R$, \emph{the density} $\delta_K(R)$   of \emph{Tri} in $B_K(R)\cap[-2,2]$  is defined to be the ratio
\begin{equation}\label{E:density}
 \delta_K(R)= \frac{|Tri\cap B_K(R)\cap[-2,2]|}{|B_K(R)\cap[-2,2]|}.
 \end{equation}
Alternatively, it might be called the relative frequency of occurrence of elements of $Tri$ in $B_K(R)\cap[-2,2]$.   If the limit
$ \lim_{R\rightarrow \infty}\delta_K(R)$ exists, we call it the \emph{density} of \emph{Tri} in $K\cap[-2,2]$, and we denote it by $\delta_K(\T)$.   It can be viewed as the probability that a randomly selected element of $K\cap[-2,2]$ belongs to $Tri$.  We  now
\vspace{.2in}

\noindent\textbf{Conjecture 1} (Main Conjecture) $\delta_K(\T)=0.$

\vspace{.2in}

Note that the definitions of $h_K$ and $B_K(R)$ depend on the choice of $\mathcal{V}$.  And so our density function  depends on this choice.  It is not hard, however, to show that the height function corresponding to another choice of basis will be commensurate to the first.  This enables us to show that if Conjecture 1 holds for one choice of basis, it will hold for any other.  We  present some details of this discussion in Appendix B.   Here we simply proceed with the definitions arising from a fixed $\mathcal{V}$.

Next, we wish to describe our computational strategy for estimating the densities (\ref{E:density}).  This will make use of some standard ``estimation language,'' which we briefly spell out for the reader's convenience.

\subsection{Estimation}\label{SS:heightandestimation}

We are interested in estimating values of real-valued functions as the arguments get large.
Usually this is done for functions with some standard domain, such as the real numbers or the integers.  However, we need to look at a broader class of domains.  Accordingly, we let $X$ be a locally-compact Hausdorff space with countable basis, and we let $\mathcal{F}$ denote the set of (not necessarily continuous) real-valued functions $f$ on $X$ such that $f^{-1}(0)$ has compact closure.

Let $f,g\in\mathcal{F}$,  with $g>0$ outside some compact set.  Then $f$ is said to be $\mathcal{O}(g)$ if the ratios $|f(x)|/g(x)$ are defined and bounded for all $x$ outside some compact set.  If $f_1$ is also in $\mathcal{F}$, such that $f-f_1$ is $\mathcal O(g)$, then we may express this by writing  $f=f_1 +\mathcal{O}(g)$.

This notation has a number of simple consequences.  For example:

\begin{enumerate}
\item If $f$ is $\mathcal{O}(g)$, $f_1, g_1\in \mathcal{F}$,  and if  $|f_1|\leq|f|$ and $g\leq g_1$, then $f_1$ is $\mathcal{O}(g_1)$.
\item If $f_i$ is $\mathcal{O}(g_i),\,\,i=1,2,\ldots, m$, then  $f_1\cdot f_2\cdot\ldots\cdot f_m$ is $\mathcal{O}(g_1\cdot g_2\cdot\ldots\cdot g_m)$.
\item If $f_i$ and $g_i$ are as in (b), and if $c_1,c_2,\ldots,c_m$ are real numbers that are not all zero, then
$\Sigma_{i=1}^mc_if_i$ is $\mathcal{O}(\Sigma_{i=1}^m|c_i|g_i)$.
\end{enumerate}

Although the ``big $\mathcal{O}$'' notation gives only a very crude connection between the values $f(x)$ and $g(x)$ as $x$ gets large in  $X$, even this can sometimes be useful.  For example, let us write $\lim_{x\rightarrow\infty}f(x)=L$, for some real number $L$, if, for each positive integer $n$, there is a compact subset $C_n$ of $X$ such that the relation  $|f(x)-L|\in[0,1/n)$ holds outside of $C_n$.  Now suppose that $f$ is $\mathcal{O}(g)$ and
$\lim_{x\rightarrow\infty}g(x)=0$. It then follows that $\lim_{x\rightarrow\infty}f(x)=0$.

\vspace{.2in}

A more refined estimation relation, namely that of asymptotic approximation, may be defined as follows.  Let $f$ and $g$ belong to $\mathcal{F}$. Then $f(x)/g(x)$ is defined outside some compact set, and the expression $\lim_{x\rightarrow\infty}(f(x)/g(x))$ makes sense.  We say that $f$ is asymptotic to $g$, written $f\sim g$, provided that $\lim_{x\rightarrow\infty}f(x)/g(x)=1$.
It is easy to check that $\sim$ defines an equivalence relation on $\mathcal{F}$.

Assuming additionally that $g>0$ outside some compact set, it is easy to check that $f\sim g$ implies that $f$ is $\mathcal{O}(g)$.  So $\sim $ is a finer relation than big $\mathcal{O}$.

However, big $\mathcal{O}$\ can be used to obtain $\sim $ under some circumstances. Namely,  choose any $f,g\in \mathcal{F}$, with $g>0$ outside some compact set,  and let $f_1=f+\mathcal{O}(g)$.  Assume that $\lim_{x\rightarrow\infty}(g(x)/f(x))=0$.
It then follows that there exist a positive constant $M$ and a compact set $C\subseteq X$ such that
\[ -Mg/f\leq 1-f_1/f \leq Mg/f\]
outside C. This implies, first, that $f_1\in \mathcal{F}$ and, second, that $f_1\sim f$.

\subsection{The Computational Strategy}\label{SS:strategy}

Although the steps that we use in working toward a proof of Conjecture 1 are mostly of an elementary computational nature, the overall structure of the argument is intricate, and so we give here a brief overview.

\subsubsection{ The numerator $|Tri\cap B_K(R)\cap[-2,2]|$ of (\ref{E:density})}\label{SSS:numerator}

The set $Tri$ appearing in the numerator of (\ref{E:density}) is not computationally easy to work with, so we replace it by a larger set that is more computationally amenable.  In particular, we use the function $f:\R\rightarrow \R$ of \S 4--- $f(x)=x^3-3x$--- and Lemma 1 of \S 4 to conclude that
\[ Tri\cap B_K(R)\cap[-2,2]\subseteq f(K)\cap B_K(R).\]
Therefore,

\begin{equation}\label{E:replacinginequalityfordensity}
\delta_K(R)\leq \frac{|f(K)\cap B_K(R)|}{|B_K(R)\cap[-2,2]|}.
\end{equation}

We shall show that the right hand side of inequality (\ref{E:replacinginequalityfordensity}) goes to zero as $R\rightarrow\infty$, so we do not lose anything by this replacement in our effort to prove Conjecture 1 .

\subsubsection{The key computations for the numerator of (\ref{E:replacinginequalityfordensity})}\label{SSS:keycomputations}

We see from equation (\ref{E:elementsofK}) that $K$ can be identified with the set of all integer $(k+1)$-tuples $(a_1,\ldots,a_k,b)$ such that $b>0$ and $a_1,\ldots,a_k,b$ are relatively prime.  It is not hard to rewrite $f$ using this identification.  $B_K(R)$ has a simple description in this notation, but a description of the  intersection $f(K)\cap B_K(R)$ involves $k+1$ polynomial inequalities in $a_1, a_2, \ldots, a_k,b$, and so it is fairly complex.

When $k$ is small, specifically, when $k\leq 2$, however, we are able to use the inequalities to find a function $S:(0,\infty)\rightarrow (0,\infty)$ such that
\begin{equation}\label{E:S}
S\quad \mbox{is}\quad \mathcal{O}(R^{1/3})
\end{equation}
and
\begin{equation}\label{E:numeratorreplacement}
f(K)\cap B_K(R)\subseteq f(B_K(S)).
\end{equation}

The proof of (\ref{E:numeratorreplacement}) follows a similar pattern for each number field $K$ but details and specific  bounds depend on the degree and discriminant of $K$.  These proofs will occupy all of \S 7.

Using the inequality for the density above, we then get
\begin{eqnarray}\label{E:newdensity}
\delta_K(R)& \leq & \frac{|f(B_K(S))|}{|B_K(R)\cap[-2,2]|}\nonumber\\
& \leq & \frac{|B_K(S)|}{|B_K(R)\cap[-2,2]|}.
\end{eqnarray}

\subsubsection{The denominator $|B_K(R)\cap[-2,2]|$.}\label{SSS:denominator}

The description of $B_K(R)$ leads immediately to an estimate for $|B_K(R)|$.
Indeed, as we shall see in \S 6.1,

\[|B_K(R)| \sim \frac{(2R)^kR}{\zeta(k+1)},\]
where $\zeta$ is the classical Riemann zeta function.  This will follow   by generalizing an argument of Sittinger that proves a theorem of Lehmer (cf. \cite{sit} and \S 5).   This asymptotic relation implies that
\begin{equation}\label{E:estimateB_K(R)}
|B_K(R)| \leq \frac{(2R)^{k+1}}{\zeta(k+1)},
\end{equation}
for sufficiently large $R$.

However, even though the subset $B_K(R)\cap[-2,2]$ is easy to describe in terms of $a_1,\ldots, a_k,b$ and the basis $\mathcal{V}$, its cardinality cannot be easily estimated.  For example, we do not have good information about how the relatively prime, positive integer $k+1$-tuples are  distributed throughout the subset  $\N^{k+1}\subset \R^{k+1}$, where \N\ is the set of natural numbers, so we cannot proceed via some sort of  volume computation.

However, since our goal is a very crude estimation, we can circumvent this problem  by defining a (relatively small) subset

\begin{equation}\label{E:C(R)}
Q(R)\subseteq B_K(R)\cap[-2,2],
\end{equation}
 for which we can prove  that
\begin{equation}\label{E:estimateC(R)}
|Q(R)|\sim \frac{2^kR^{k+1}}{(k+1)^{k+1}||\mathcal{V}||\zeta(k+1)}.
\end{equation}
Here $||\mathcal{V}||$ is a positive constant depending only on the basis $\mathcal{V}$---
in particular, \emph{not} on $R$.  This will occupy \S 6.2.  It follows from (\ref{E:estimateC(R)}) that
\begin{equation}\label{E:lowerboundC(R)}
|Q(R)|\geq \frac{2^{k-1}R^{k+1}}{(k+1)^{k+1}||\mathcal{V}||\zeta(k+1)},
\end{equation}
for sufficiently large $R$.

Therefore, using (\ref{E:newdensity}), (\ref{E:estimateB_K(R)}),  (\ref{E:C(R)}), and (\ref{E:lowerboundC(R)}), we have
\begin{eqnarray}
\delta_K(R)& \leq &  \frac{|B_K(S)|}{|Q(R)|}\nonumber\\
 &\leq & \frac{(2S)^{k+1}(k+1)^{k+1}||\mathcal{V}||}{2^{k-1}R^{k+1}}\nonumber\\
&=&4||\mathcal{V}||\left(\frac{(k+1)S}{R}\right)^{k+1}.
\end{eqnarray}
Since $S$ is $\mathcal{O}(R^{1/3})$, by (\ref{E:S}),

\begin{equation}
\left(\frac{(k+1)S}{R}\right)^{k+1}\quad\mbox{is}\quad\mathcal{O}(R^{-\frac{2}{3}(k+1)}),
\end{equation}
which implies  the following result.

\begin{theorem}\label{T:maintheorem}
 Let $K$ be a real number field of degree  $k\leq 2$.  Then,
\begin{equation}
\delta_K(R) \quad\mbox{is}\quad \mathcal{O}(R^{-\frac{2}{3}(k+1)}).
\end{equation}
\end{theorem}

This clearly implies Theorem 1.1.  We state its extension to all real number fields:

\begin{con}: $\delta_K(R)$ is $\mathcal{O}(R^{-\frac{2}{3}(k+1)})$, for any real, degree $k$ extension $K$ of \Q.
\end{con}

\vspace{.2in}

We now begin the proofs of the above results.
\section{A generalization of Lehmer's Theorem}\label{S:LehmersTheorem}

Choose integers $k$ and $n$, with $k\geq 2$ and $n\geq 1$, and let $Q(k,n)$ be the set of all relatively prime $k$-tuples of positive integers $\leq n$.  As before, we make use of the Riemann zeta function $\zeta$.

\begin{theorem}[D. Lehmer, 1900]\label{T:Lehmer}  Fix the integer $k$.  Then
\[ |Q(k,n)|\sim \frac{n^k}{\zeta(k)}.\]
\end{theorem}

In more recent work \cite{sit}, B.D. Sittinger has shown that

\begin{equation}\label{E:Sittinger}
|Q(k,n)|= \frac{n^k}{\zeta(k)} + \mathcal{O}(f_k(n)),
\end{equation}
where $f_k(n)=n\ln(n)$, when $k=2$, and $f_k(n)=n^{k-1}$, when $k>2$.
Equation (\ref{E:Sittinger}) immediately implies Theorem \ref{T:Lehmer} (cf. \S 4.2).

We will generalize (\ref{E:Sittinger}) in two ways.  First, we  replace $n$ by a $k$-tuple $\mathbf{n}=(n_1,\ldots, n_k)$, and second, we allow each $n_i$ to be an \emph{arbitrary real} number $\geq 1$.  Let $Q(k,\mathbf{n})$ be the set of all relatively prime $k$-tuples of positive integers $(a_1,\ldots,a_k)$ such that each $a_i\leq n_i$, where $\mathbf{n}$ satisfies the conditions just given.

We think of $Q(k,\mathbf{n})$ as a generalized ``cube'' with sides of length $n_i$.  The set $Q(k,n)$ appearing in (\ref{E:Sittinger}) represents the case in which all side-lengths equal a given positive integer $n$.

It will now be convenient to introduce the notion of \emph{eccentricity} of $Q(k,\mathbf{n})$.
We define this as follows:

\begin{equation}
e(Q(k,\mathbf{n}))= \frac{\max\{n_1,\dots,n_k\}}{\min\{n_1,\dots,n_k\}}.
\end{equation}

Clearly, $e(Q(k,\mathbf{n}))\geq 1$, with equality holding if and only if all $n_i$ are equal.
For any real number $E\geq 1$, let $\mathcal{C}^k_E$ denote the set of all $\mathbf{n}\in \R^k$\, such that each $n_i\geq 1$ and $e(Q(k,\mathbf{n}))\leq E$.  This set inherits a locally-compact topology from $\R^k$; we may call it the \emph{space of $k$-cubes of eccentricity $\leq E$}.

Next, we let $\gamma(\mathbf{n})$ denote the geometric mean of the the $n_i$ comprising $\mathbf{n}$, i.e.,

\begin{equation}
\gamma(\mathbf{n})=(n_1\cdot\ldots\cdot n_k)^{1/k}.
\end{equation}
We use this to define a function $f_k(\mathbf{n})$ as follows:

\begin{equation}
f_k(\mathbf{n})=\left\{\begin{array}
                {r@{\quad:\quad}l}
                \gamma(\mathbf{n})\ln(\gamma(\mathbf{n})) & k=2\\
                \gamma(\mathbf{n})^{k-1} & k>2.
                \end{array}\right.
\end{equation}

Clearly this gives one reasonable way to generalize the definition of the same-named function appearing in (\ref{E:Sittinger}).

We can now state the desired generalization:

\begin{theorem}\label{T:generalizedSittinger}  Fix $k$ and choose a real number $E\geq 1$.  Then, for $\mathbf{n}$ ranging over
$\mathcal{C}^k_E$, we have
\[ |Q(k,\mathbf{n})|=\frac{n_1\cdot\ldots\cdot n_k}{\zeta(k)} + \mathcal{O}(f_k(\mathbf{n})).\]
\end{theorem}

%We discuss possible further generalizations at the end of this section.

An easy computation that follows directly from the definitions shows that
\[ \lim_{\mathbf{n}\rightarrow\infty}(f_k(\mathbf{n})/n_1\cdot\ldots\cdot n_k)=0.\]
According to \S 4.2, Theorem \ref{T:generalizedSittinger} then implies that
\begin{equation}
|C(k,\mathbf{n})|\sim \frac{n_1\cdot\ldots\cdot n_k}{\zeta(k)}.
\end{equation}

The remainder of this section is devoted to a proof of Theorem \ref{T:generalizedSittinger}.

\subsection{The integral case}\label{SS:integralcase}

We  begin by proving an analog of the theorem in which $\mathbf{n}$ ranges over the integral $k$-tuples in $\mathcal{C}_E^k$. That is, $\mathbf{n}$ ranges over $\mathcal{C}_E^k\cap\Z^k$. The proof follows that of Sittinger's proof of (\ref{E:Sittinger}), with modifications to take into account the fact that not all the $n_i$ are equal.

To simplify the notation in the computation, we adopt the following convention: whenever we have a $k$-tuple of reals, say $(z_1,\ldots,z_k)$, we shall write $\pi(z_i)$ to denote the product $z_1\cdot\ldots\cdot z_k$.

Using the inclusion-exclusion principle, we compute

\begin{equation}\label{E:inclusionexclusion}
|Q(k,\mathbf{n})|=\pi(n_i) -\sum_{p_1}\pi([n_i/p_1]) + \sum_{p_1<p_2}\pi([n_i/p_1p_2])-\sum_{p_1<p_2<p_3}\pi([n_i/p_1p_2p_3])+\ldots,
\end{equation}
where $[\quad]$ denotes the greatest integer function and the $p_i$ range over the set of primes. Note that each of the terms $[n_i/p_1\dots p_r]$ is zero when either $r$ is sufficiently large or some $p_i$ is sufficiently
large.  Therefore the expression on the right hand side reduces to a finite sum. As Sittinger does in his special case,
we consolidate (\ref{E:inclusionexclusion}) by using the M\"{o}bius function $\mu$:

\begin{equation}\label{E:consolidatedwithmobius}
|Q(k,\mathbf{n})|= \sum_{j=1}^{\infty}\mu(j)\pi([n_i/j]).
\end{equation}
Clearly the summands for which $j>\min\{n_1\ldots, n_k\}$ all vanish.We now need a lemma to help evaluate the products $\pi([n_i/j])$.

\begin{lemma}\label{L:estimatelemma} For any $\mathbf{x}=(x_1,\ldots,x_k)$ in
$\mathcal{C}^k_E$, set
\[\phi(\mathbf{x})=\pi(x_i)/\min\{x_1,\ldots,x_k \}.\]
Let $\mathbf{y}=\mathbf{y}(\mathbf{x})$ be any function $\mathcal{C}^k_E\rightarrow \mathcal{C}^k_E$ satisfying $|x_i-y_i|\leq 1$ for all $i=1,\ldots,k$.
  Then
\[ \pi(x_i)=\pi(y_i)+\mathcal{O}(\phi(\mathbf{x})).\]
\end{lemma}

We give a proof at the end of this section.
\begin{cor}\label{C:corollarytoestimatelemma} Let $\mathbf{x}$ and $\mathbf{y}$ be as in Lemma \ref{L:estimatelemma}.
Then,
\[\pi(x_i)=\pi(y_i)+\mathcal{O}(\gamma(\mathbf{x})^{k-1}).\]
\end{cor}
\begin{proof}Choose any $\mathbf{y}$ as in Lemma \ref{L:estimatelemma}.  The lemma implies that
\begin{equation}\label{E:estimatelemmaequation}
\frac{|\pi(x_i)-\pi(y_i)|}{\phi(\mathbf{x})}
\end{equation}
is bounded for all $\mathbf{x}\in \mathcal{C}^k_E$. Now, $\max\{x_1,\ldots,x_k\}/E\geq \gamma(\mathbf{x})/E$.  Therefore, by our eccentricity assumption, $\min\{x_1,\ldots,x_k\}\geq \max\{x_1,\ldots,x_k\}/E\geq\gamma(\mathbf{x})/E$. It follows that $\phi(\mathbf{x})\leq E\pi(x_i)/\gamma(\mathbf{x})=E\gamma(\mathbf{x})^{k-1}$.  Combining this with (\ref{E:estimatelemmaequation},) we conclude that
\[ \frac{|\pi(x_i)-\pi(y_i)|}{\gamma(\mathbf{x})^{k-1}}\]
is bounded, as desired.
 \end{proof}

We now return to our proof of Theorem \ref{T:generalizedSittinger} by applying Corollary \ref{C:corollarytoestimatelemma} to the case in which
\[ x_i=\frac{n_i}{j}\quad\mbox{and}\quad y_i=\left[\frac{n_i}{j}\right],\]
assuming that $j\leq\min\{n_1,\ldots,n_k\}$.
Note that
\[\gamma(n_1/j,\ldots,n_k/j)=\gamma(n_1,\ldots,n_k)/j.\]
Therefore, \ref{C:corollarytoestimatelemma} yields
\begin{equation}
\pi([n_i/j])=\pi(n_i/j) +\mathcal{O}(\gamma(\mathbf{n})^{k-1}/j^{k-1}).
\end{equation}

Feeding this into equation (\ref{E:consolidatedwithmobius}) (and recalling that we may assume that $j\leq\min\{n_1,\ldots,n_k\}$ and that the M\"{o}bius function $\mu$ assumes only the values $0$ and $\pm 1$ ), we get

\begin{eqnarray}\label{E:evalC(k,n)1}
|Q(k,\mathbf{n})| & = & \sum_{j=1}^m\mu(j)\pi(n_i)/j^k + \sum_{j=1}^{m}\mu(j)(\pi([n_i/j])-\pi(n_i/j))\nonumber\\
& = & \sum_{j=1}^m\mu(j)\pi(n_i)/j^k + \mathcal{O}(\sum_{j=1}^m|\mu(j)|(\gamma(\mathbf{n})^{k-1}/j^{k-1})\nonumber\\
& = & \pi(n_i)\sum_{j=1}^m\mu(j)/j^k + \mathcal{O}(\gamma(\mathbf{n})^{k-1}\sum_{j=1}^m1/j^{k-1}),
\end{eqnarray}
where $m=\min\{n_1,\ldots,n_k\}$.

Next, Sittinger observes that, by the definition of the zeta function,
\[ \sum_{j=1}^m\mu(j)/j^k= 1/\zeta(k) - \sum_{j=m+1}^{\infty}\mu(j)/j^k,\]
with the tail dominated by
\[ \int_m^{\infty}\frac{dt}{t^k}=\frac{1}{(k-1)m^{k-1}}.\]
(Recall that throughout this section, we are assuming that $k\geq 2$.)

Therefore, setting $M=\max\{n_1,\ldots,n_k\}$,

\begin{eqnarray}\label{E:evalC(k,n)2}
\pi(n_i)\sum_{j=m+1}^{\infty}\mu(j)/j^k & \leq & \frac{m\pi(n_i)}{(k-1)m^k}\nonumber\\
&\leq & \frac{m\pi(n_i)E^k}{(k-1)M^k}\nonumber\\
&\leq & \frac{E^km}{k-1}\nonumber\\
&\leq & \frac{E^k}{k-1}\cdot\gamma(\mathbf{n}).
\end{eqnarray}

Next,  we compute
\[ \mathcal{O}(\gamma(\mathbf{n})^{k-1}\sum_{j=1}^m1/j^{k-1})\]
by observing that
\[\sum_{j=1}^m\frac{1}{j^{k-1}}\;\;\mbox{is\;dominated\;by}\;1+\int_1^m\frac{dt}{t^{k-1}},\]
which is
\[ \left\{\begin{array}{l} \mathcal{O}(\ln(\gamma(\mathbf{n})),\quad k=2,\\
                            \mathcal{O}(1),\quad k> 2.
                            \end{array}\right . \]

Therefore,
\begin{equation}\label{E:evalC(k,n)3}
\mathcal{O}(\gamma(\mathbf{n})^{k-1}\sum_{j=1}^m1/j^{k-1})\quad\mbox{is}\quad
\left\{\begin{array}{l}  \mathcal{O}(\gamma(\mathbf{n})\ln(\gamma(\mathbf{n})),\quad k=2,\\
                            \mathcal{O}(\gamma(\mathbf{n})^{k-1}),\quad k> 2.
                            \end{array}\right .
\end{equation}

Combining (\ref{E:evalC(k,n)1}), (\ref{E:evalC(k,n)2}), and (\ref{E:evalC(k,n)3}), we obtained the desired result when $\mathbf{n}$ ranges over $\mathcal{C}_E^k\cap\Z^k$.

Our next task is to derive from this the result for general $\mathbf{n}$.

\subsection{ The case of general real $k$-tuple $\mathbf{n}$}\label{SS:generalcase}

Choose any $\mathbf{n}$ in $\mathcal{C}_{E}^k$ and set
\[[\mathbf{n}]=([n_1],\ldots,[n_k]).\]
It is not hard to check that $[\mathbf{n}]\in \mathcal{C}_{2E}^k\cap\Z^k$.

By what was proved in the integral case,
\[|Q(k,\mathbf{m})| = \frac{\pi(m_i)}{\zeta(k)} + \mathcal{O}(f_k(\mathbf{m})),\]
for $\mathbf{m}$ ranging over $\mathcal{C}_{2E}^k\cap\Z^k.$

It will be convenient to reformulate this as follows:  \emph{There exists a constant $H$, depending only on $E$ and $k$, such that}
\begin{equation}\label{E:integralcase}
|\;|Q(k,\mathbf{m})|- \frac{\pi(m_i)}{\zeta(k)}| \leq H\cdot f_k(\mathbf{m}),
\end{equation}
\emph{for $\mathbf{m}$ ranging over $\mathcal{C}_{2E}^k\cap\Z^k$}.

Clearly, $[n_i]\leq n_i$, so that $f_k([\mathbf{n}])\leq f_k(\mathbf{n})$. Further, it is immediate from the definition that  $Q(k,[\mathbf{n}])=Q(k,\mathbf{n})$.   Therefore,  replacing $\mathbf{m}$ by $[\mathbf{n}]$ in   inequality  (\ref{E:integralcase})  and using the foregoing observations, we get
\begin{equation}\label{E:nearlydone}
|\;|Q(k,\mathbf{n})|- \frac{\pi([n_i])}{\zeta(k)}| \leq  H\cdot f_k(\mathbf{n}).
\end{equation}

Finally, by Corollary \ref{C:corollarytoestimatelemma}, there is a constant $H'$, depending only on $k$, such that
\[ |\pi(n_i) - \pi([n_i])|\leq H'\cdot\gamma(\mathbf{n})^{k-1}.\]

Dividing this last inequality by $\zeta(k)$, adding the result   to (\ref{E:nearlydone}) and using the definition of $f_k(\mathbf{n})$, it follows that

\[|\;|Q(k,\mathbf{n})|- \frac{\pi(n_i)}{\zeta(k)}|\leq H''\cdot f_k(\mathbf{n}),\]
where $H''$ is a constant depending only on $E$ and $k$.  This translates to the desired statement involving big $\mathcal{O}$.

\vspace{.1in}
It remains to prove Lemma \ref{L:estimatelemma}.

\subsection{Proof of Lemma \ref{L:estimatelemma}}\label{SS:proofofestimatelemma}

We begin with the  identity
\begin{equation}\label{E:generalidentity}
\pi(x_i)-\pi(y_i)=\sum_{h=1}^k(x_h-y_h)y_1\cdot\ldots\cdot y_{h-1}x_{h+1}\cdot\ldots\cdot x_{k},
\end{equation}
which can be proved by induction on $k$ or by a simple algebraic manipulation.
Equation (\ref{E:generalidentity}) holds for all elements $x_i, y_j$ in any  commutative ring.

Next, since the hypothesis of the lemma states that $|x_i-y_i|\leq 1$, for all $i$,  (\ref{E:generalidentity}) implies that
\begin{equation}
|\pi(x_i)-\pi(y_i)|\leq \sum_{h=1}^k(1+x_1)\cdot\ldots\cdot(1+x_{h-1})\cdot x_{h+1}\cdot\ldots\cdot x_k.
\end{equation}

We may write
\[(1+x_1)\cdot\ldots\cdot(1+x_{h-1})= 1+\sigma_1^{h-1}(x_1,\dots,x_{h-1})+\ldots+\sigma_{h-1}^{h-1}(x_1,\dots,x_{h-1}),\]
where $\sigma_a^{h-1}$ is the $a^{th}$ elementary symmetric function in $h-1$ variables.
Therefore,
\[ |\pi(x_i)-\pi(y_i)|\leq \sum_{h=1}^k\sum_{a=0}^{h-1}\sigma_a^{h-1}(x_1,\dots,x_{h-1})x_{h+1}\cdot\ldots\cdot x_k.\]
Now $\sigma_a^{h-1}$ is the sum of all products of $a$ distinct unknowns selected from the $h-1$
unknowns.  So
\[ \sigma_a^{h-1}(x_1,\dots,x_{h-1})x_{h+1}\cdot\ldots\cdot x_k \]
consists of ${h-1\choose a}$ terms, each of the form
\begin{equation}\label{E:typicalterm}
x_{i_1}\cdot\ldots\cdot x_{i_a}\cdot x_{h+1}\cdot\ldots\cdot x_k.
\end{equation}
Let $\{j_1,\ldots,j_b\}$ denote the complement of $\{i_1,\ldots,i_a\}$ in $\{1,2,\ldots, h-1\}$, where $a+b=h-1$.  Then, we may rewrite expression (\ref{E:typicalterm}) as
\begin{equation}\label{E:typicaltermrewritten}
\frac{x_1\cdot\ldots\cdot x_{h-1}\cdot x_{h+1}\cdot\ldots\cdot x_k}{x_{j_1}\cdot\ldots\cdot x_{j_b}}.
\end{equation}

Since each $x_i\geq 1$, the expression in (\ref{E:typicaltermrewritten}) is $\leq \phi(\mathbf{x})$.  Therefore,
\begin{eqnarray*}
|\pi(x_i)-\pi(y_i)| & \leq & \sum_{h=1}^k\sum_{a=0}^{h-1}{h-1\choose a}\phi(\mathbf{x})\\
& = &\phi(\mathbf{x})\left(\sum_{h=1}^k2^{h-1}\right)\\
& = & (2^k-1)\phi(\mathbf{x}).
\end{eqnarray*}

This immediately implies Lemma \ref{L:estimatelemma} and, with it, completes the proof of
Theorem \ref{T:generalizedSittinger}.

\section{Bounds on the numerators and denominators of the density estimate}

\subsection{Estimating the cardinality of $B_K(R)$}\label{SS:estimating|B_K(R)|}\mbox{}\\

Recall that $B_K(R)$ consists of all relatively prime integer $(k+1)$-tuples $(a_1,\dots,a_k,b)$ such that each $|a_i|\leq R$ and $0<b\leq R$.  In this section, we obtain the following estimate:

\begin{equation}\label{E:estimateofB_K(R)}
|B_K(R)| = \frac{2^kR^{k+1}}{\zeta(k+1)} +  \mathcal{O}(f_{k+1}(R)).
\end{equation}

Since
\[\lim_{R\rightarrow\infty}\frac{f_{k+1}(R)}{R^{k+1}}=0,\]
it follows that (cf. the last paragraph in \S {\ref{SS:heightandestimation})
\begin{equation}
|B_K(R)|\sim \frac{2^kR^{k+1}}{\zeta(k+1)}.
\end{equation}

Our derivation of the estimate (\ref{E:estimateofB_K(R)}) is based on Theorem~\ref{T:generalizedSittinger}.  However, that theorem refers only to tuples whose entries are positive integers.  So, we must see how to include zero and negative entries into our count.  To do this we first introduce some extra notation.

\vspace{.1in}

Let $I$ and $J$ be disjoint subsets of $\{1,\ldots,k\}$.

\vspace{.1in}

Define
\begin{equation}
B_K(R;I,J;k) = \left\{(a_1,\ldots,a_k,b)\in B_K(R):\begin{array}{ccl}
 a_i  <  0 & \Leftrightarrow & i\in I\\
 a_i = 0 & \Leftrightarrow & i\in J\\
 a_i > 0 & \Leftrightarrow & \mbox{otherwise}.
\end{array}\right\}.
\end{equation}
\vspace{.1in}

In the notation of \S \ref{S:LehmersTheorem},
\[ B_K(R;\emptyset,\emptyset;k)=Q(k+1, \mathbf{R}),\]
where $\mathbf{R}$ is the $(k+1)$-tuple $(R,\ldots,R)$.

Therefore, according to Theorem~\ref{T:generalizedSittinger},
\begin{equation}\label{E:basicboxestimate}
|B_K(R;\emptyset,\emptyset;k)|=\frac{R^{k+1}}{\zeta(k+1)}+ \mathcal{O}(f_{k+1}(\mathbf{R})),
\end{equation}
where
\begin{equation}
f_{k+1}(\mathbf{R})= \left\{\begin{array}{cc}
                                R\ln(R), & k=1\\
                                R^k, & k>1.
                                \end{array}\right.
\end{equation}

For each subset $I$, the set $B_K(R;I,\emptyset;k)$ consists entirely of relatively prime \mbox{$(k+1)$-tuples} $(a_1,\ldots,a_k,b)$ for which all the entries are non-zero.  Since changing the sign of one or more of the $a_i$'s does not affect their absolute values or divisibility properties, such sign changes can be used to
define a bijection between any two of the $B_K(R;I,\emptyset;k)$'s.  Of course, they are all pairwise disjoint.  So, using equation~(\ref{E:basicboxestimate}), we obtain
\begin{equation}\label{E:coreestimate}
|\;\bigsqcup_IB_K(R;I,\emptyset;k)|=\frac{2^kR^{k+1}}{\zeta(k+1)}+\mathcal{O}(f_{k+1}(\mathbf{R}),
\end{equation}
where $\bigsqcup$ denotes the disjoint sum.

\vspace{.1in}

The remaining $(k+1)$-tuples in $B_K(R)$ consist of those for which some $a_i$'s are zero. These all sit in ``lower dimensional'' cubes, and so their contribution gets absorbed by the big $\mathcal{O}$\ notation. We make this precise as follows.

Suppose first that $k>1$, $0<m<k$, and $J$ is a subset of $\{1,\ldots,k\}$ of cardinality $m$.  Then delete the elements of $J$ from $\{1,\ldots,k\}$, and renumber the remaining numbers, in order, using $\{1,\ldots,k-m\}$.  Given any $I$ disjoint from $J$ as before, renumber it using the renumbering just obtained.  This produces a subset $I'$ of \\$\{1,\ldots,k-m\}$.  These operations on indices determine a bijection between $B_K(R;I,J;k)$ and $B_K(R;I',\emptyset;k-m)$.  Therefore, for any fixed, non-empty $J$ of cardinality $m$, equation (\ref{E:coreestimate}) implies that

\begin{equation}~\label{E:bigsqcupB_K}
|\;\bigsqcup_IB_K(R;I,J;k)|=
\frac{2^{k-m}R^{k-m+1}}{\zeta(k-m+1)}+\mathcal{O}(f_{k-m+1}(\mathbf{R}).
\end{equation}

\vspace{.1in}

If $J^{\ast}$ is any \underline{other} non-empty subset of $\{1,\ldots,k\}$, then $\bigsqcup_IB_K(R;I,J;k)$ and $\bigsqcup_{I^{\ast}}B_K(R;I^{\ast},J^{\ast};k)$ are disjoint.  They have the same cardinality when $|J|=|J^{\ast}|$. Therefore, letting $J$ range over all non-empty \emph{proper} subsets of $\{1,\ldots,k\}$, we have

\begin{equation}~\label{E:doublebigsqcup} |\bigsqcup_{J\neq\emptyset}\bigsqcup_IB_K(R;I,J;k)| = \sum_{|J|=m=1}^{k-1}\left({k\choose m}\frac{2^{k-m}R^{k-m+1}}{\zeta(k-m+1)}+ \mathcal{O}(f_{k-m+1}(\mathbf{R})\right).
\end{equation}

We leave to the reader the check that the expression on the right is $\mathcal{O}(f_{k+1}(\mathbf{R}))$.

Now consider the case $m=k$.  Then $J=\{1,\ldots, k\}$, and
the only possible set $I$ is the empty set.  In this case the left-hand side of  equation (\ref{E:bigsqcupB_K}) reduces to $|B_K(R;\emptyset,J)|$. But the cube  $B_K(R;\emptyset,J)$  is just the singleton set consisting of $(0,\ldots,0,1)$, and so, allowing the case $J=\{1,\ldots,k\}$ in the expression on the left-hand side of (\ref{E:doublebigsqcup}), we still get that its cardinality is  $\mathcal{O}(f_{k+1}(\mathbf{R}))$

A similar special argument applies to the case $k=1$, which we leave to the reader.
\vspace{.1in}

We can now conclude: Since $B_K(R)$ is precisely the disjoint union of $\;\bigsqcup_IB_K(R;I,\emptyset;k)$ and $\bigsqcup_{J\neq\emptyset}\bigsqcup_IB_K(R;I,J;k)$, estimate (\ref{E:estimateofB_K(R)}) follows immediately.

\vspace{.2in}

\subsection{Defining Q(R) and estimating its cardinality}\label{SS:estimating|C(R)|}

We recall from Section~\ref{S:densityoverview} that the definition of density as well as all the related concepts and computations began with a choice of basis $\mathcal{V}=\{v_1,\ldots,v_k\}$ of the field $K$ over \Q. We  made  no assumptions about $\mathcal{V}$. It will now be convenient, for notational and computational simplicity, to make the assumption that each real number $v_i$ is $\geq 1$ (cf. Appendix B).  Set $||\mathcal{V}||=\pi(v_i)\; (=v_1\cdot\ldots\cdot v_k)$.

Define $\mathbf{m}$ and $\mathbf{n}$ in $\R^{k+1}$ as follows:
\begin{eqnarray*}
m_i  =  n_i & = & \frac{2R}{(k+1)v_i},\quad i=1,\ldots,k\\
m_{k+1} & = & \frac{k}{k+1}R,\\
n_{k+1} & = & R.
\end{eqnarray*}

Then the set $Q(R)$ is defined to be the set difference
\begin{equation}\label{E:definitionofC(R)}
Q(R)=Q(k+1,\mathbf{n})\setminus Q(k+1,\mathbf{m}),
\end{equation}
where we use the ``cubes'' defined in Section~\ref{S:LehmersTheorem}. Since $Q(k+1,\mathbf{m})\subseteq Q(k+1,\mathbf{n})$, we have
\begin{equation}
|Q(R)|=|Q(k+1,\mathbf{n})|- |Q(k+1,\mathbf{m})|.
\end{equation}

According to Theorem~\ref{T:generalizedSittinger} and the definition of $Q(k+1,\mathbf{n})$,
\begin{equation}\label{E:$n$}
|Q(k+1,\mathbf{n})|= \frac{2^kR^{k+1}}{(k+1)^k||\mathcal{V}||\zeta(k+1))}+\mathcal{O}(f_{k+1}(\mathbf{n})).
\end{equation}

Applying Theorem~\ref{T:generalizedSittinger} to $|Q(k,\mathbf{m})|$, and using the fact that $f_{k+1}(\mathbf{m})\leq f_{k+1}(\mathbf{n})$, we get, similarly, that
\begin{equation}\label{E:$m$}
|Q(k+1,\mathbf{m})|= \frac{k2^kR^{k+1}}{(k+1)^{k+1}||\mathcal{V}||\zeta(k+1)}+\mathcal{O}(f_{k+1})(\mathbf{n}),
\end{equation}
where we think of $\mathbf{n}$ as a function of $\mathbf{m}$.

Therefore, combining (\ref{E:$n$}) and (\ref{E:$m$}),
\begin{eqnarray*}
    |Q(R)| & = & \frac{2^kR^{k+1}}{(k+1)^k||\mathcal{V}||\zeta(k+1)}-\frac{k2^kR^{k+1}}
    {(k+1)^{k+1}||\mathcal{V}||\zeta(k+1)}+\mathcal{O}(f_{k+1}(\mathbf{n}))\nonumber\\
 & = & \frac{2^kR^{k+1}}{(k+1)^{k+1}||\mathcal{V}||\zeta(k+1)}+\mathcal{O}(f_{k+1}(\mathbf{n})).
 \end{eqnarray*}
An easy computation shows that the geometric mean $\gamma(\mathbf{n})$ is given by
\[ \gamma(\mathbf{n})=R\cdot \left(\frac{2^k}{(k+1)^k||\mathcal{V}||}\right)^{\frac{1}{k+1}},\]

so that, setting $D$ equal to the coefficient of $R$ in this expression, we get
\[f_{k+1}(\mathbf{n})=\left\{\begin{array}{ll}
                    DR\ln(DR),& k=1\\
                    D^kR^k,& k>1.
                    \end{array}\right.\]

Therefore, we obtain
\begin{equation}\label{E:|C(R)|}
|Q(R)|=\frac{2^kR^{k+1}}{(k+1)^{k+1}||\mathcal{V}||\zeta(k+1)}+ \mathcal{O}(F_{k+1}(R)),
\end{equation}
where, here, $F_{k+1}(R)$ is obtained from $f_{k+1}(\mathbf{n})$ above by deleting all reference to the
constant factor $D$.

Again, as before, we obtain from the above big $\mathcal{O}$ relation the corresponding asymptotic relation

\begin{equation}\label{E:asymptoticestfor|C(R)|} |Q(R)|\sim\frac{2^kR^{k+1}}{(k+1)^{k+1}||\mathcal{V}||\zeta(k+1)}.
\end{equation}

\vspace{.2in}

\subsection{Using $|Q(R)|$ as a lower bound for $|B_K(R)\cap[-2,2]|$}

\begin{lemma}\label{L:C(R)inB} $Q(R)\subseteq B_K(R)\cap[-2,2]$. Hence, $|Q(R)|\leq |B_K(R)\cap[-2,2]|$
\begin{proof}  Referring to the defining equation for $Q(R)$ (equation (\ref{E:definitionofC(R)})),
we note that since $Q(k+1,\mathbf{n})$ is a subset of $B_K(R)$, by construction, we need only check that $(a_1,\ldots,a_k,b)$ in $Q(R)$ satisfies
\[-2\leq \frac{a_1v_1+\ldots+a_kv_k}{b}\leq 2.\]
Moreover, since all the terms in the middle expression are positive, it remains only to verify the right-hand inequality.

Choose any $(a_1,\ldots,a_k,b)$ in $Q(R)$.  Then, by construction,
\[    0  <  a_i  \leq  \frac{2R}{(k+1)v_i},\quad i=1,\ldots,k\]
and
\[    \frac{k}{k+1}R  <  b  \leq  R.\]
Therefore,
\[0 < \frac{a_1v_1+\ldots+a_kv_k}{b}\leq \frac{\frac{2k}{k+1}R}{\frac{k}{k+1}R}= 2,\]
as desired.\end{proof}
\end{lemma}

\vspace{.1in}

We may now use Lemma~\ref{L:C(R)inB}, together with the asymptotic estimate (\ref{E:asymptoticestfor|C(R)|}), to get a lower bound for $|B_K(R)\cap[-2,2]|$.  In particular, choose any $\epsilon\in (0,1)$. Then,

\[|B_K(R)\cap[-2,2]|\geq\frac{2^{k-\epsilon}R^{k+1}}
{(k+1)^{k+1}||\mathcal{V}||\zeta(k+1)},\]

for $R$ sufficiently large.  This is clearly a vast underestimate in general, but it will do for our purposes.

\section{Proof of Theorem \ref{T:maintheorem}}

Recall that Theorem~\ref{T:maintheorem} asserts that
\[ \delta_K(R)\quad\mbox{is}\quad \mathcal{O}(R^{-\frac{2}{3}(k+1)})\]
when $K$ is a real field of degree $k\leq 2$. As shown in Section~\ref{S:densityoverview}, in the presence of the estimates in the preceding section, this follows from the existence of a function
\[ S:(0,\infty)\rightarrow (0,\infty),\]
such that

\begin{enumerate}
\item $S$ is $\mathcal{O}(R^{\frac{1}{3}})$, and
\item $f(K)\cap B_K(R)\subseteq f(B_K(S))$.
\end{enumerate}
Recall that $f$ here is the polynomial function given by $f(x)= x^3-3x.$

\vspace{.1in}

In this section we construct such a function $S$. In general, $S$ will depend on $K$,  although the basic form and idea of the construction will be the same for each $K$.

\vspace{.1in}

Note that a typical element in the left-hand set in b) above is of the form $f(\alpha)$ such that the height
$h_K(f(\alpha))$ is $\leq R$.  In order to gain usable information from this fact, we must be able to compute this height or some bound on the height in terms of the data supplied by $\alpha$.  The problem we initially face is that, for any $\beta\in K$,  $h_K(\beta)$ is defined in terms of a canonical representation of $\beta$ in terms of the selected basis $\mathcal{V}$ of $K$ (cf. (2)).  The $ k+1$ integers appearing in this representation are assumed to be relatively prime.  However, although this is what we may assume for the  integers appearing in the representation of $\alpha$, when we apply $f$ to this representation and expand to get the result into the appropriate form, the  integer coefficients we get need not be relatively prime.  Our first task, therefore, is to obtain a bound on the greatest common divisor of these coefficients.

\subsection{Bounding the greatest common divisor}\label{SS:boundingthegcd}

When $K=\Q$, there is no problem.  For if we choose $a/b\in\Q$, where $a$ and $b$ are relatively prime integers and $b>0$, then $f(a/b)=(a^3-3b^2a)/b^3$, and it is easy to check that numerator and denominator are relatively prime.  So we now turn to real fields of degree $2$.

\vspace{.1in}

Real quadratic fields are known to be of the form $\Q(\sqrt{d})$, where $d$ is any positive, square-free integer. In this case, we choose the basis $\mathcal{V}$ to be the set $\{1, \sqrt{d}\}$.  $\mathcal{V}$ consists of integral elements of $K$, but we do not use this fact.  Every $\alpha$ in $K$ may be written uniquely as
\begin{equation}\label{E:integralrep}
\alpha=\frac{a_1+a_2\sqrt{d}}{b},
\end{equation}
where $a_1,a_2, b$ are relatively prime integers and $b>0$ (cf. (\ref{E:elementsofK})).  Now apply $f$ to (\ref{E:integralrep}) to obtain
\begin{equation}\label{E:applyfandexpand}
f(\alpha)=\alpha^3-3\alpha=\frac{(a_1^3+3da_1a_2^2-3a_1b^2)+(3a_1^2a_2+da_2^3-3a_2b^2)\sqrt{d}}{b^3}.
\end{equation}
In this subsection, we write the long expression as
\[ \frac{A_1+A_2\sqrt{d}}{B}\]
to simplify notation.  Often, we shall use the triple $(A_1,A_2,B)$ instead of this fraction.
Let $G$ denote the greatest common divisor (g.c.d.) of $A_1, A_2, B$.

Using this notation, we can express  the height $h_K(f(\alpha))$ as follows:
\begin{equation}{\label{E:height}}
 h_K(f(\alpha))=\frac{\max(|A_1|,|A_2|,B)}{G}.
\end{equation}

\vspace{.2in}

\begin{lemma}\label{L:gcdbound}
$G|8d$
\end{lemma}
\begin{proof}
Let $p$ be a prime dividing $G$, and suppose that $p|a_1$.  Since $p|B$, we know that $p|b$, and so we cannot also have $p|a_2$.  Therefore, using $p|A_2$, we have $p|3a_1^2+da_2^2-3b^2$.  This implies $p|d$.

Next, suppose that $p^2|G$ and also $p|a_1$.  Then, since $p^2$ divides $3a_1^2a_2-3a_2b^2$ as well as $A_2$, we have $p^2|da_2^3$.  We still cannot have $p |a_2$ from the above argument, so $p^2|d$, which contradicts the fact that $d$ is squarefree.

Therefore, any common prime factor $p$ of $G$ and $a_1$ must be a factor of $d$ and occurs only to the first power in $G$.

Now suppose that a prime $p$ divides $G$ but $p$ does not divide $a_1$.  In this case, $p$ divides $A_1/a_1=a_1^2+3da_2^2-3b^2$ and also $b$, so $p$ cannot divide $a_2$.  This implies that $p$ divides $A_2/a_2= 3a_1^2+da_2^2-3b^2$.  Hence $p$ divides both
\[a_1^2+3da_2^2\]
and
\[3a_1^2+da_2^2,\]
which implies that $p|8a_1^2$, hence $p|8$.  Therefore, in this case $p=2$.

Still sticking to the case $p|G$ and $p\not\mid a_1$ (so $p=2$), suppose that $2^4|G$.  Since $a_1$ and $a_2$ are both odd in this case, and odd numbers represent invertible elements in the ring of integers $\mod 16$, we may divide $A_1$ by $a_1$ and $A_2$ by $a_2$\ in that ring to obtain congruences
\[\begin{array}{ccc}
    a_1^2+3da_2^2-3b^2 & \equiv & 0\mod 16\\
    3a_1^2+da_2^2-3b^2 & \equiv & 0\mod 16.
    \end{array}\]
By our assumption on $G$, we have $16|B=b^3$, which implies that $4|b$, hence $b^2\equiv 0\mod{16}$.  Therefore,
the above equations become
\[\begin{array}{ccc}
    a_1^2+3da_2^2 & \equiv & 0\mod 16\\
    3a_1^2+da_2^2 & \equiv & 0\mod 16.
    \end{array}\]
Subtracting the first of these from three times the second, we get
\[8a_1^2\equiv 0\mod 16,\]
a contradiction since $a_1$ is odd.

Therefore, the highest power of $2$ dividing $G$ is $\leq2^3$.

The result is now immediate.
\end{proof}

Applying the lemma to equation (\ref{E:height}), we get

\begin{cor}\label{C:lowerboundheight} \[h_K(f(\alpha))\geq \frac{\max(|A_1|,|A_2|,B)}{8d}.\]
\end{cor}

\vspace{.2in}

\subsection{A certain cubic curve}

The estimates that we want to make to conclude the proof of Theorem~\ref{T:maintheorem}  all involve features of a certain  cubic function:
\[\Phi_{D,E}(x)=D(x^3-3E^2x),\]
where $D$ and $E$ are positive real parameters.

\begin{lemma}\label{L:cubiccurve}
Choose any real $T>0$ and suppose that $E\leq T^{1/3}$.  If $x\geq T^{1/3}+E$, then $\Phi_{D,E}(x)>DT$.

Therefore, making use of the contrapositive,  $\Phi_{D,E}(x)\leq DT\quad\Rightarrow\quad x\leq 2T^{1/3}.$
\end{lemma}

The proof is an exercise in elementary calculus and so will be omitted.

Note that since $\Phi_{D,E}$ is an odd function of $x$,  Lemma~\ref{L:cubiccurve} implies that for $T$ and
$E$ as in the lemma,
\[ \Phi_{D,E}(x)\geq -DT\quad\Rightarrow\quad x\geq -2T^{1/3}.\]

\subsection{Constructing S}

We continue with the notation of Section~\ref{SS:boundingthegcd}

\subsubsection{The case K=\Q}.  Recall that, for $a/b\in\Q$, $a,b$ relatively prime and $b>0$, we have
\[f(a/b)=\frac{a^3-3b^2a}{b^3}=\Phi_{1,b}(a)/b^3.\]
Therefore, applying Lemma~\ref{L:cubiccurve}(c), with $T=R$, $D=1$, and $E=B$, we may conclude that if
$\Phi_{1,b}(a)\leq R$ and $b\leq R^{1/3}$, then $a\leq 2R^{1/3}$.  Similarly, by the remark following the lemma,
if $\Phi_{1,b}(a)\geq -R$ and $b\leq R^{1/3}$, then $a\geq -2R^{1/3}$.  Using the fact that the numerator and denominator in the above expression for $f(a/b)$ are relatively prime, we may interpret the foregoing as saying that
\[f(a/b)\in B_K(R)\Rightarrow a/b\in B_K(2R^{1/3}).\]
Now apply $f$ to the right hand side of this implication to conclude that
\[f(K)\cap B_K(R)\subseteq f(B_K(2R^{1/3})).\]

\vspace{.1in}

This argument shows that we may define the desired function $S$ by
\[ S(R)=2R^{1/3},\]
thus concluding the proof of Theorem~\ref{T:maintheorem} in the case $K=\Q$.

\subsubsection{The case of real quadratic fields}

We refer to Section~\ref{SS:boundingthegcd} and particularly to expression (\ref{E:applyfandexpand}) to point to the notation that we shall be using here.

\vspace{.1in}

\noindent\textbf{(a)} We assume throughout this part that $B\leq 8dR$ so that $b=B^{1/3}\leq (8dR)^{1/3}$, where $R$ is an arbitrary positive real as before. (Here $8dR$ will correspond to the real number $T$ appearing in Lemma~\ref{L:cubiccurve} and $b$ will correspond to the parameter E.)

Suppose now that $a_1\geq (8dR)^{1/3}+b$.  In particular, $a_1$ is positive, so we have  inequality

\begin{equation*}
A_1=a_1^3+3da_1a_2^2-3a_1b^2\geq a_1^3-3b^2a_1=\Phi_{1,b}(a_1).
\end{equation*}

Using Lemma~\ref{L:cubiccurve}, we get the further inequality

\begin{equation*}
\Phi_{1,b}(a_1)>8dR.
\end{equation*}

Therefore, (always assuming $B\leq 8dR$), we get the implication
\[a_1\geq (8dR)^{1/3}+b\quad\Rightarrow\quad A_1 > 8dR.\]

Using the contrapositive version of this,  we conclude that  we have the implication
\begin{equation}
A_1=a_1^3+3da_1a_2^2-3a_1b^2\leq 8dR\quad\Rightarrow\quad a_1\leq(8dR)^{1/3}+b \leq 2(8dR)^{1/3}.
\end{equation}

Similarly, using the fact that $\Phi_{1,b}(X)$ is an odd function of $x$, as mentioned after Lemma~\ref{L:cubiccurve}, we may also conclude that when $b\leq (8dR)^{1/3}$,
we have the implication
\begin{equation}
A_1=a_1^3+3da_1a_2^2-3a_1b^2\geq -8dR\quad\Rightarrow\quad a_1\geq -2(8dR)^{1/3}.
\end{equation}

Consolidating these, we get
\[\left.\begin{array}{ccc}\label{E:A1boundimplies}
B & \leq & 8dR\\
|A_1| & \leq & 8dR
\end{array}\right\}\quad\Rightarrow\quad |a_1|\leq 2(8dR)^{1/3}.\]

\vspace{.1in}

\noindent\textbf{(b)}\quad We now apply a similar argument to
\[A_2=3a_1^2+da_2^3-3a_2b^2.\]
We assume throughout this part that $B\leq 8R$. (Here $8R$ will correspond to the real number $T$
appearing in Lemma~\ref{L:cubiccurve} and $c=b/\sqrt{d}$ will correspond to the parameter $E$.)

Suppose that $a_2\geq (8R)^{1/3}+c$.
Of course $a_2$ is positive, so that we get the inequality
\[A_2=3a_1^2+da_2^3-3a_2b^2> da_2^3-3a_2b^2=d(a_2^3-3c^2a_2)=\Phi_{d,c}(a_2),\]
We can now apply Lemma~\ref{L:cubiccurve} again to  get the further inequality
\[\Phi_{d,c}(a_2)> 8dR.\]
Thus, as above, we get an implication
\[a_2\geq (8R)^{1/3}+c\quad\Rightarrow\quad A_2>8dR.\]
A similar argument yields, in the presence of the assumption $B\leq 8dR$,
\[a_2\leq -(8R)^{1/3}-c\quad\Rightarrow\quad A_2<-8dR.\]
Combining these two implications and passing to contrapositives, we get
\[\left.\begin{array}{ccc}\label{E:A2boundimplies}
B & \leq & 8dR\\
|A_2| & \leq & 8dR
\end{array}\right\}\quad\Rightarrow\quad |a_2|\leq 2(8R)^{1/3}<2(8dR)^{1/3}.\]

\vspace{.1in}

\noindent\textbf{(c)}\quad We now wrap things up and finish the proof of Theorem~\ref{T:maintheorem}.

\vspace{.1in}

Suppose that $f(\alpha)\in B_K(R)$. That is, $h_K(f(\alpha))\leq R$.  By equation (\ref{E:height})
and Corollary \ref{C:lowerboundheight}, we get
\[\frac{\max(|A_1|,|A_2|,B)}{8d}\leq\frac{\max(|A_1|,|A_2|,B)}{G}=h_K(f(\alpha))\leq R,\]
so that
\[\max(|A_1|,|A_2|,B)\leq 8dR.\]
The conclusions in (a) and (b) above imply that $h_K(\alpha)=\max(|a_1|,|a_2|,b)\leq 2(8dR)^{1/3}$.
Therefore,  $\alpha\in B_K(2(8dR)^{1/3})$, implying that $f(\alpha)\in f(B_K(2(8dR)^{1/3}))$, i.e.,
$f(K)\cap B_K(R)\subseteq f(B_K(2(8dR)^{1/3}))$.

The desired conclusion now follows by defining the function $S=S(R)$ by
\[ S(R)= 2(8dR)^{1/3}.\]
This completes the proof of Theorem~\ref{T:maintheorem}.

\section{Some further comments}

Several further directions are possible for the inquiry begun by this paper.

\vspace{.1in}

For example,
one could attempt to prove Conjecture 1 for other real number fields $K$ or for all of them. And one could attempt to prove the sharper Conjecture 2 for fields of low degree.  There is also
the global question, presumably substantially harder, of obtaining the density of $Tri$ in $\A\cap\R$.

\vspace{.1in}

Another kind of problem would be to improve the estimates given in this paper, even for fields $K$ of
small degree. If one follows the arguments given here, it seems that what is required is further
information on the ``geometric'' distribution of relatively prime $k$-tuples of integers. For example, I do not know whether the distribution is uniform throughout $\R^k$.  Perhaps analytic number theorists have looked at this, but I do not know of any such results.

\vspace{.1in}

Along another line, we mentioned briefly that the set of  angles that are both constructible and trisectable forms a countable subgroup of the circle group.  It would be interesting to obtain further information about this group.

\vspace{.1in}

Finally, one might attempt to solve similar estimation problems for $p$-sectability, either for various
specific primes $p$ or for primes $p$ in general. (See Appendix A below for some preliminary results in this direction.)  Since the key trigonometric equations are much
more complicated for  $p\geq 3$, other techniques are probably required.

\vspace{.3in}

\begin{appendix}

\section{n-sectability of angles}\label{S:n-sectability}

  Let $n$ be any positive integer.  The reader may naturally wonder what $n$-fold subdivisions are achievable for all angles via Euclidean ruler and compass construction. Of course, when  $n$ has the form $2^k$, such subdivisions are always possible via iterated bisection.  The case $n=3$ was settled by P-L. Wantzel, as we have discussed in the main body of this paper.  The case of general $n$ is settled by the following theorem:

\begin{theorem}  Suppose $n$ is a positive integer such that, for any angle $\alpha$,  there exists a Euclidean  construction that starts with  $\alpha$ and produces $\alpha/n$, i.e.,  suppose that every $\alpha$ is $n$-sectable.   Then, $n$ has the form $2^k$, for some non-negative integer $k$.
\end{theorem}

\begin{proof}

\noindent\textbf{Case 1:}\quad Assume that $n$ is an odd prime.  We write $n=p$.

The standard identity $(\exp(ip\theta)=\exp(i\theta)^p$ yields the following trigonometric formula:
\[ cos(p\theta)= \sum_{k=0}^q\sum_{\ell=0}^k(-1)^{k+\ell}{p\choose 2k}{k\choose \ell}\cos(\theta)^{p-2k+2\ell},\]
where $q=\frac{1}{2}(p-1)$. Set $a=\cos(p\theta)$ and $x=\cos(\theta)$.  Then, we have an equation

\begin{equation}\label{E:psectionpoly}
P(x,a)= \sum_{k=0}^q\sum_{\ell=0}^k(-1)^{k+\ell}{p\choose 2k}{k\choose \ell}x^{p-2k+2\ell}\;-\;a=0.
\end{equation}

We regard $P(x,a)$ as a polynomial in $x$ with parameter the constant term $a$, analogous to the polynomial $p(x,a)$ defined in \S 1.2.  Indeed, the explicit relationship between $P(x,a)$ and $p(x,a)$ when $n=3$ is given by
$2P(x,a)= p(2x,2a)$.
We now obtain information about the coefficients of $P(x,a)$ in equation (\ref{E:psectionpoly}).

\vspace{.1in}

\noindent\textbf{a)}\quad The top-degree monomial in $P(x,a)$ is $2^{p-1}x^p$.

\vspace{.1in}

To see this, we consider the summands on the right hand side of equation (\ref{E:psectionpoly}), and, holding $k$-fixed, we see that the maximum exponent obtainable occurs when $\ell=k$, yielding $x^p$.  This is independent of $k$, so $x^p$ is the maximum power of $x$ occurring in the formula.  This would imply that the degree of $P(x,a)$ is $p$, provided that the coefficients in the formula satisfying $\ell=k$ do not sum to zero.

The terms with $\ell=k$ have coefficient sum $\sum_{k=0}^q(-1)^{2k}{p\choose 2k}= \sum_{k=0}^q{p\choose 2k}$, which is clearly not zero.
The symmetry properties of the terms ${p\choose 2k}$ allow us to compute the sum.  For consider the terms ${p\choose m}, m=0, 1,2, \ldots, p$.  These pair off as equals
${p\choose m}\leftrightarrow {p\choose p-m},$
with $m$ even if and only if $p-m$ is odd.  It follows that  $\sum_{k=0}^q{p\choose 2k}= \frac{1}{2}\sum_{m=0}^p{p\choose m}=2^{p-1}$.

This verifies that $P(x,a)$ has degree $p$ with leading coefficient $2^{p-1}$.

\vspace{.1in}

\noindent\textbf{b)}\quad  Except for the leading coefficient and the parameter $a$, every coefficient in $P(x,a)$ is divisible by $p$.  Indeed the coefficient of the first-degree term is $(-1)^qp$.

\vspace{.1in}

The coefficient $(-1)^{k+\ell}{p\choose 2k}{k\choose\ell}$ is clearly divisible by the prime $p$ as long as $k\neq 0$.
This implies the first statement.  For the second statement, set $p-2k+2\ell=1$.  Then it follows that $0\leq \ell=k-\frac{1}{2}(p-1)=k-q\leq 0$.  Therefore $\ell=0$ and $k=q$, which implies that the coefficient of $x$ in $P(x,a)$ is as stated.

\vspace{.1in}

We now can apply the  Eisenstein criterion \cite{wae} to the polynomial $P(x,a)$ for appropriate choice of the
parameter $a$.  In particular, choose  any integer $c$ that is divisible by $p$ but not by $p^2$, and let $d$
be any positive integer prime to $c$ such that that $-1\leq c/d\leq 1$.  Then $d^pP(x,c/d)$ is a polynomial in $\Z[x]$ whose top coefficient  is not divisible by $p$, whose  remaining coefficients are divisible by $p$, but whose constant term is not divisible by $p^2$.  These are precisely the conditions under which the Eisenstein criterion implies that $d^pP(x,c/d)$ is irreducible in $\Z[x]$, except possibly for constant factors. It follows immediately from Gauss's Lemma that $P(x,c/d)$ is irreducible in $\Q[x]$.

The argument now is almost identical to the case $n=3$ argued before. Choose $\alpha$ such that $\cos(\alpha)=c/d$. Then $\cos(\alpha/p)$ is a zero of $P(x,c/d)$, by equation (\ref{E:psectionpoly}) and the preceding trigonometric formula. Suppose $\cos(\alpha/p)$ were constructible over $\{(0,0), (1,0), \alpha\}$, and let $f$ be its minimal polynomial.  Then the degree of $f$ over $\Q(\cos(\alpha))=\Q(c/d)=\Q$\ is divisible by a power of $2$ and $f$ is a factor of $P(x,c/d)$ in $\Q[x]$, contradicting the irreducibility of $P(x,c/d)$. Therefore, $\alpha/p$ is not constructible, concluding Case 1.

\vspace{.2in}

\noindent\textbf{Case 2:} General n.  We begin with a simple general observation.

\begin{lemma}\label{L:nsectableimpliesksectable}  Suppose that $k$ is a factor of $n$.  If $\alpha$ is $n$-sectable, then it is $k$-sectable.
\end{lemma}
\begin{proof}  Write $n=k\ell$.  Since $\alpha/n$ is constructible over $\{(1,0),(0,1),\alpha\}$, so is the multiple $\alpha/k=\ell\alpha/n$.
\end{proof}

Now suppose that $n$ is not a power of $2$.  Then it has an odd prime factor
$p$. Let $\alpha$ be an angle that is not $p$-sectable,  which exists by Case 1. Then, by Lemma \ref{L:nsectableimpliesksectable}, $\alpha$ is not $n$-sectable.  This proves the theorem.
\end{proof}

\vspace{.2in}

\noindent\textbf{Remarks:}\quad (a)  The above argument can be slightly elaborated to imply that if $r$ is a rational number strictly between $0$ and $1$, and if, for any given angle $\alpha$, there is a Euclidean  construction that starts with $\alpha$ and produces  the angle $r\alpha$, then $r$ must have the form $k/2^{\ell}$, for some integers $k$ and $\ell$, where $\ell\geq0$.  Of course, when $r$ does have that form and any $\alpha$ is given, the angle $r\alpha$ \underline{can} be constructed.\vspace{.1in}

(b) Let $\alpha$ be any angle, and set $a=\cos(\alpha)$ as before.  Choose any positive integer $n$.  Then, as we have seen above $\cos(\alpha/n)$ is algebraic over $\Q(a)$.  In particular, $\cos(\alpha/n)$ is an  algebraic number whenever $a$ is.  It follows that  the non-$n$-sectable angles $\alpha$ produced in the proof of Theorem A.1 have algebraic cosines.\vspace{.1in}

\begin{prop}If $n$ is a positive integer such that $2\pi/n$ can be constructed (i.e., the regular polygon of $n$ sides can be constructed), then there exists a countable dense subset of $S^1$ consisting of $n$-sectable angles.
 \end{prop}
 \begin{proof}This follows from the construction of Yates described earlier in \S 3.1.  Namely, let $m$ be any positive integer prime to $n$, and choose integers $a$ and $b$ such that $an+bm=1$.  Multiply this equation by the
quantity $2\pi/mn$  Then, $(1/n)(2\pi/m)= a(2\pi/m)+b(2\pi/n)$, showing that $2\pi/m$ is $n$-sectable.
The set of such angles $2\pi/m$ is clearly countable and dense in $S^1$.
\end{proof}\vspace{.1in}

As is well known, Gauss showed that a necessary and sufficient condition for $2\pi/n$ to be constructible is that $\phi(n)$ be a power of $2$, where $\phi$ is the Euler function.  Examples of odd $n$ for which this is true are: $n= 3,5,17,257,65537. $\vspace{.1in}

\begin{prop} If $n$ is a positive odd integer such that $2\pi/n$ can be constructed , then there exists a countable dense subset of $S^1$ consisting of non-$n$-sectable angles.
\end{prop}
\begin{proof}To see this,  suppose that $n$ is an odd number such that $2\pi/n$ is constructible, and suppose that $\beta$ is a non-$n$-sectable angle.
By the argument for Proposition A.1, every integer multiple of $\pi/2^k$ is $n$-sectable, for any $k\geq 0$. \emph{We claim that $\gamma=\beta + c\pi/2^k$ is not $n$-sectable, for every integer $c$ and every $k\geq 0$}.  The argument is essentially that given in \S 1.2.  We give it here for the reader's convenience: Suppose $\gamma$ were $n$-sectable.  Then, starting with $\beta$ we could construct $\gamma$ and then $\gamma/n$.  Construct $c\pi/n2^k$ and subtract this from $\gamma/n$, obtaining $\beta/n$.  This would provide a Euclidean $n-section$ of $\beta$, which is impossible. The set of all these non-$n$-sectable $\gamma$'s is clearly countable and dense in $S^1$.
\end{proof} \vspace{.1in}

If $\cos(\beta)$ above is an algebraic number---and such $\beta$ exist for every $n$ that is not a power of $2$, by Remark (b) above---then it is easy to see that $\cos(\gamma)$ is algebraic.  Therefore, the foregoing gives a countable dense set of non-$n$-sectable angles with algebraic cosines.  The case of transcendental cosines is handled by the next result, which extends Corollaries \ref{C:transcendentalimpliesirreducible} and \ref{C:cardinalityoftrisectibleangles}.\vspace{.2in}

\begin{theorem}  Suppose that $cos(\alpha)$ is transcendental and that $n$ is a positive integer that is not a power of $2$.  Then $\alpha$ is not $n$-sectable.  Therefore, the set of all non-$n$-sectable angles is uncountable, and the set of all $n$-sectable angles is countable.
\end{theorem}
\begin{proof}  We use the notation introduced above.  In particular, we use the polynomial $P(x,a)$ in
(\ref{E:psectionpoly}), where $a=\cos(\alpha)$.  Let $P(x,t)\in \Q[t]$ be the polynomial obtained from $P(x,a)$ by replacing $a$ by the indeterminate $t$.

Assume first that $n$ is an odd prime $p$.  We claim that $P(x,t)$ is irreducible in $\Q[t]$.
To see this, choose $c/d$ as in Case 1 of the proof of Theorem A.1.  Let $\chi$ be the $\Q$-algebra homomorphism that sends $t$ to $c/d$, so that $\chi(P(x,t))=P(x,c/d)$.  Since $P(x,c/d)$ is irreducible, by Case 1 of the proof of Theorem A.1, $P(x,t)$ must be irreducible, as claimed.

It follows immediately that $P(x,a)$ is irreducible, because $\Q[t]\cong\Q[a]$ when $a$ is transcendental.

We now argue as we did earlier.  Since $n$ is not a power of $2$, it has an odd prime factor $p$. If $\cos(\alpha/p)$ is constructible over $\Q(a)$, then its minimal polynomial over $\Q(a)$, say $f$, has degree a power of $2$.  In particular, the degree does not equal $0$ or $p$.  But $\cos(\alpha/p)$ is a zero of $P(x,a)$, so $f$ divides $P(x,a)$, contradicting the irreducibility of $P(x,a)$.

Therefore, $\alpha$ is not $p$-sectable.  Applying Lemma \ref{L:nsectableimpliesksectable}, we conclude that $\alpha$ is not $n$-sectable.

The last two statements of the theorem simply use the standard facts about transcendental numbers and algebraic numbers.
\end{proof}

\vspace{.2in}

\section{Change of basis}

Let $\mathcal{V}_1 = \{v_1,\ldots,v_k\}$ and $\mathcal{V}_2 = \{w_1,\ldots,w_k\}$ be \Q-vector space bases of the real number field $K$, and let $h_1$ and $h_2$ be the corresponding height functions, as defined in \S 7.1.  We show first that $h_1$ and $h_2$ are commensurate.  More precisely, we show that there is a positive integer $d$ such that
\begin{equation}\label{E:commensurability}
\frac{1}{d}h_2\leq h_1 \leq dh_2.
\end{equation}

Let $T=[t_{ij}]$ be the $k\times k$ matrix of rational numbers given by
\[ w_j=\sum_{i=1}^k t_{ij}v_i, \quad j=1,\ldots, k.\]

We say that $T$ is \emph{elementary} if one of the following is true: (a) $T$ represents a permutation of the basis $\mathcal{V}_1$. (b) $T$ represents the addition (resp., subtraction) of one basis vector of $\mathcal{V}_1$ to (resp., from) another. (c) $T$ represents the multiplication of one basis vector of $\mathcal{V}_1$ by a non-zero rational number while fixing the others.

Of course, every product of elementary matrices is invertible. Elementary row-reduction would imply the converse, except row reduction allows a slightly richer class of elementary matrices of type (b). However, it is easy to see that these can be obtained by multiplying suitable elementary matrices of the above type. So, every invertible matrix is a product of ones that are elementary in the above sense.

\begin{lemma} If $T$ is an elementary matrix, then  (\ref{E:commensurability}) holds.
\end{lemma}
\begin{cor}\label{C:changeofbasis} Let $h_1$ and $h_2$ be height functions corresponding to two choices of rational bases of $K$.  Then, there exists a positive integer $d$ satisfying the inequalities (\ref{E:commensurability}).
\end{cor}
\begin{proof} Let $T$ be the change of basis matrix, factor it into a product of elementary matrices, and apply the lemma successively to these.
\end{proof}

It remains to prove the lemma.

\begin{proof} (a)  When $T$ is a permutation matrix, the representations of  a field element $\alpha$ in terms of the two bases differ only by a permutation  of the coefficients.  But the definition of the height function shows that it is invariant under permutation of coefficients, so that $h_1=h_2$ in this case: i.e., the inequalities are satisfied for $d=1$.

(b) Without loss of generality, let us assume that the basis change involves adding (resp., subtracting) $v_2$ to (resp. from) $v_1$ and fixing the other vectors.
Then, writing
\begin{equation}\label{E:representationofalpha}
\alpha= (a_1v_1+\ldots+a_kv_k)/b,
\end{equation}
as in \S 7.1, we  have
\[\alpha= (a_1w_1+(a_2 \mp a_1)w_2+a_3w_3+\ldots+a_kw_k)/b.\]
It is easy to see that $a_1,\ldots,a_k,b$ are relatively prime if and only if $a_1,(a_2\mp a_1), a_3,\ldots,a_k,b$ are relatively prime.  Therefore, assuming the former,we get
\[h_2(\alpha)=\max\{|a_1|, |a_2\mp a_1|,|a_3|,\ldots,|a_k|,b\}\leq 2\max\{|a_1|,\ldots,|a_k|,b\}=2h_1(\alpha).\]
A symmetric argument proves the same inequality with $h_1$ and $h_2$ exchanged. Therefore (\ref{E:commensurability}) holds in this case with $d=2$.

(c) Here we do not lose generality by assuming that $t_{ij}=\delta_{ij}$ (the Kronecker delta)
when $(i,j)\neq (1,1)$, and $t_{11}=t$,  where $t$ is a rational number, which can be written ``in lowest terms'' as $t=r/s$.  Then, with $\alpha$ as above in (\ref{E:representationofalpha}), we may write
\[\alpha =(a_1/t)w_1+a_2w_2+\ldots a_k w_k)/b=((sa_1)w_1+(ra_2)w_2+\ldots+(ra_k)w_k)/rb.\]

Notice that  in this last representation of $\alpha$,  the integers $sa_1, ra_2, \ldots, ra_k, rb$ may not be relatively prime.  So, we cannot apply the usual formula for the height function to these.  However,
if we divide all these integers by their greatest common divisor, $G$, then the resulting integers are relatively prime, and they do result from a representation of $\alpha$.  It follows that
\begin{equation}\label{E:equationforh2}
h_2(\alpha)=\max\{sa_1, ra_2, \ldots, ra_k, rb\}/G.
\end{equation}

To proceed further, we need some sort of upper bound for $G$.  This is provided by the following claim:
\emph{$G$ divides $rs$}.  To see this, suppose that $p$ is a prime dividing $G$ and that $p^m$ is the highest power of $p$ dividing $G$.  If $p^m$ fails to divide $r$ \emph{and} $p^m$ fails to divide $s$, then $p$ must divide $a_1, \ldots, a_k$ and $b$, a contradiction.  Therefore, $p^m$ divides $r$, or it divides $s$. Hence it divides $rs$.  It follows that $G|rs.$

We now apply this to equation (\ref{E:equationforh2}).
\begin{eqnarray*}
h_2(\alpha) & \geq & \max\{sa_1, ra_2, \ldots, ra_k, rb\}/|rs|\\
&\geq & \min\{|r|,|s|\}\max\{|a_1|,\ldots, |a_k|, b\}/|rs|\\
&= & h_1(\alpha)/\max\{|r|,|s|\}.
\end{eqnarray*}

Thus, the right-hand inequality in (\ref{E:commensurability}) holds in this last case as well, with $d=\max\{|r|,|s|\}.$  A symmetric argument produces the left-hand inequality.
\end{proof}

\vspace{.2in}

Next we use  Corollary \ref{C:changeofbasis} to show that our conjectures and results concerning density do not depend on the choice of the basis.

\vspace{.1in}

If $h_1$ and $h_2$ are height functions satisfying (\ref{E:commensurability}), and if
\begin{equation}
 B_i(R)=h_i^{-1}[0,R),
 \end{equation}
then Corollary \ref{C:changeofbasis} implies that
\begin{equation}\label{E:boxrelation}
B_1(R/d)\subseteq B_2(R)\subseteq B_1(dR).
\end{equation}

We recall the definition of density, with respect to each height function:
\begin{equation}\label{E:delta}
\delta_i(R)=\frac{|Tri\cap B_i(R)|}{|[-2,2]\cap B_i(R)|},
\end{equation}
for $i=1,2,$ (cf. \S 7.1).  We shall show that $\delta_1(R)$  is $\mathcal{O}(R^{-\frac{2}{3}(k+1)})\Leftrightarrow \delta_2(R)$ is $\mathcal{O}(R^{-\frac{2}{3}(k+1)})$.

Of course, by symmetry we need only prove one implication, say $\Rightarrow$.

Using (\ref{E:delta}) for $i=2$, together with (\ref{E:boxrelation}), we have

\[ \delta_2(R)\leq\frac{|Tri\cap B_1(dR)|}{[-2,2]\cap B_2(R)|}.\]
Using (\ref{E:boxrelation}) again, we get
\[ \delta_2(R)\leq\frac{|Tri\cap B_1(dR)|}{[-2,2]\cap B_1(R/d)|}= \delta_1(dR)\cdot\frac{|[-2,2]\cap B_1(dR)|}{|[-2,2]\cap B_1(R/d)|}. \]

We now use the ``box'' $Q_1(R/d)$ constructed exactly as $Q(R)$ is constructed in \S 6.2, equation (41).  The  inclusion
$Q_1(R/d)\subseteq [-2,2]\cap B_1(R/d)$ follows exactly as in Lemma 6.1, so we get
\[\delta_2(R)\leq \delta_1(dR)\cdot\frac{|B_1(dR)|}{|Q_1(R/d)|}.\]
Just as in \S 6, one shows that both $|B_1(dR)|$ and $|Q_1(R/d)|$ are $\mathcal{O}(R^{k+1})$.  It follows that their quotient is bounded, say by $M$.

Thus, we have shown that $\delta_2(R)\leq M\delta_1(dR)$.  From this, the desired implication is immediate.

\vspace{.2in}

\section{Constructible non-trisection numbers of arbitrarily high degree}

Since the standard examples of  non-trisectable angles---namely, $\pi/3+\pi/2^n$--- are constructible, it may be of some interest to obtain  information about how complicated they are to construct, by which we mean the minimal number of Euclidean ruler and compass steps  it would take to construct them.  This problem is certainly not well-posed, but even without going into a lengthy analysis, we can probably agree
that the $\log_2$ of the  algebraic degree of $\cos(\pi/3+\pi/2^n)$ (or, $2\cos(\pi/3+\pi/2^n)$) gives a weak lower bound.  It ignores all the ``rational'' constructions required, but it does count the ``quadratic'' ones.

\vspace{.1in}

In this appendix we prove the following:

\begin{prop}  The degree of $2\cos(\pi/3+\pi/2^n)$ is $2^n$.
\end{prop}

The proof is an extended exercise that uses standard trigonometric identities and well-known facts about field extensions.

\subsection{Basic identities and computations}

Define the numbers $a_n,b_n,c_n,d_n$ as follows, for all $n\geq 0$:

\begin{eqnarray}
a_n & = & 2\cos(\pi/2^n)\\
b_n & = & 2\sin(\pi/2^n)\\
c_n & = & 2\cos((\pi/3)+(\pi/2^n))\\
d_n & = & 2\cos((\pi/3)-(\pi/2^n))
\end{eqnarray}

It is easy to prove, say inductively, that these numbers are all algebraic.  In fact, since the angles in question are all obviously constructible, so are their sines and cosines (and also the doubles of these).
Therefore, their degrees must be powers of $2$.  Our task is to show that these powers are not lower than expected.

Next, we display certain standard trigonometric identities in terms of the numbers $a_n,b_n,c_n, d_n$.
These will be used in our arguments.  We also give a table of values of these numbers for $n\leq 2$.

\begin{eqnarray}
a_{n-1} & = & a_n^2-2.\\
a_{n-1} & = & 2-b_n^2.\\
b_{n-1} & = & a_nb_n.\\
c_n & = & \frac{1}{2}a_n-\frac{\sqrt{3}}{2}b_n.\\
d_n & = & \frac{1}{2}a_n+\frac{\sqrt{3}}{2}b_n.\\
d_{n-1} & = & 2-c_n^2.\\
a_{n-1} & = & c_nd_n+1.\\
c_{n-1} & = & 2-d_n^2.
\end{eqnarray}

\vspace{.2in}
\begin{center}
\begin{math}
\begin{array}{|c||c|c|c|c|}\hline
n & a_n & b_n & c_n & d_n\\
\hline\hline
0 & -2 & 0 & -1 & -1\\
\hline
1 & 0 & 2 & -\sqrt{3} & \sqrt{3}\\
\hline
2 & \sqrt{2} & \sqrt{2} & \frac{1-\sqrt{3}}{\sqrt{2}} & \frac{1+\sqrt{3}}{\sqrt{2}}\\
\hline
\end{array}
\end{math}
\end{center}

\vspace{.2in}

\subsection{The degrees of $a_n$ and $b_n$}

In addition to $\Q$, it will be convenient to work with the field $K=\Q(\sqrt{3})$.  This has class number $1$. That is, its ring of integers $\mathcal{O}_K$ is a $UFD$.  So, every irreducible in $\mathcal{O}_K$ is a prime.  Eisenstein's Theorem applies to $\mathcal{O}_K$.

\vspace{.1in}

We define polynomials $p_n(x)\in \Z[x],\quad n\geq 1$:
\begin{eqnarray*}
p_1(x) & = & x^2 -2\\
p_n(x) & = & p_1(p_{n-1}(x)).
\end{eqnarray*}

\begin{lemma}\label{L:valuesofpolynomial} For all $1\leq k\leq n$,
\begin{eqnarray*}
p_k(a_n) & = & a_{n-k}\\
p_k(b_n) & = & \left\{ \begin{array}{c}
                        -a_{n-1},\quad k=1\\
                        a_{n-k},\quad k\geq 2.
                        \end{array}\right .
\end{eqnarray*}
\end{lemma}

We omit the easy induction proof.

\begin{lemma} $p_n(x)$ is irreducible over \Q\ and over $K$.
\end{lemma}
\begin{proof}

We show first that $p_n(x)$ has the form $x^{2^n}+2xq(x)\pm 2$,
for some $q(x)\in \Z[x]$.  When $n=1$ this is immediate from the definition.  Assume the
result for $n-1\geq 1$ and compute
\[ p_n(x)=p_{n-1}(x)^2-2= x^{2^n}+2x(2q(x)x^{2^{n-1}}\pm 2x^{2^{n-1}-1}+2xq(x)^2\pm 4q(x))+2,\]
which has the desired form.

Now consider first the case of \Q.  We use the prime $2\in \Z$, and we apply Eisenstein's criterion to $p_n(x)$, which clearly satisfies it.  Thus $p_n(x)$ is irreducible over \Q.

For the case $K=\Q(\sqrt{3})$, we use the prime $1+\sqrt{3}\in \mathcal{O}_K$.  (To see that $1+\sqrt{3}$ is irreducible in $\mathcal{O}_K$, compute the norm $N(1+\sqrt{3})=(1+\sqrt{3})(1-\sqrt{3})=-2$, which is prime in \Z.  Since $\mathcal{O}_K$ is a UFD, $1+\sqrt{3}$ is prime.)  Again, Eisenstein's criterion is seen to be satisfied. So $p_n(x)$ is irreducible over $K$. \end{proof}

By Lemma \ref{L:valuesofpolynomial}, $p_{n-1}(a_n)=p_{n-1}(b_n)=a_1=0$, for $n\geq 2$.  Therefore, using a direct calculation to take care of the case $n=1$, we have:

\begin{cor}\label{C:degreean} $deg_{\Q}(a_n)= deg_{\Q}(b_n)=deg_{K}(a_n)=deg_{K}(b_n)=2^{n-1}$, for $n\geq 1$.
\end{cor}

\vspace{.2in}

\subsection{The fields $F(a_n), F(b_n), F(c_n), F(d_n)$, for $F=\Q, K$.} \quad From identity (64), we may conclude that
\begin{equation}\label{E:F(an)increasing}
F(a_n) < F(a_{n+1}),
\end{equation}
for all $n\geq 0$ and $F=\Q$\ or $K$.  Identity (66), together with  $F(a_0)=F(a_1)=F$ and Corollary \ref{C:degreean}, allows one to prove inductively that $F(a_n)=F(b_n)$, for all $n$.  We leave this to the reader.  Using this and (\ref{E:F(an)increasing}), we get
\begin{equation}\label{E:F(b_n)increasing}
F(b_n) < F(b_{n+1}),
\end{equation}
for all $n\geq 0$.   Each of the extensions in (\ref{E:F(an)increasing}) and (\ref{E:F(b_n)increasing})
has degree $2$, for $n\geq 1$, by Corollary \ref{C:degreean}.

Equations (69) and (71) yield two infinite towers of field extensions
\[ K= K(d_1)< K(c_2) < K(d_3) < \ldots\]
and
\[ K= K(c_1)< K(d_2) < K(c_3) < \ldots,\]
where each extension has degree $\leq 2$.

\begin{lemma} For all $n\geq 0$, $K(c_n)=K(d_n)=K(a_n)=K(b_n)$.
\end{lemma}
\begin{proof} We already have the last equality. The proof of the remaining equalities is by induction on $n$.  The cases $n=0, 1$ are obvious, using the chart of computed values.  Assume the result for $n-1$. Since, $K(a_n)=K(b_n)$, equations (67) and (68) imply that
$c_n, d_n \in K(a_n)$.  Thus, we have
\[ K(a_{n-1})=K(c_{n-1})=K(d_{n-1})< K(c_n), K(d_n) < K(a_n).\]
Therefore, since $[K(a_n):K(a_{n-1})]=2$, it suffices to show that $c_n, d_n\not\in K(a_{n-1})$.  But, if $c_n\in K(a_{n-1})$, then equation (70) would show that $d_n\in K(a_{n-1})$. So, adding identities (67) and (68), we could conclude that $a_n\in K(a_{n-1})$, which is not possible.  Therefore, $c_n\not\in K(a_{n-1})$.  Similarly, $d_n\not\in K(a_{n-1})$. This concludes the induction.
\end{proof}

Combining this with Corollary \ref{C:degreean}, we get

\begin{cor}\label{C:degreecn}  $deg_K(c_n)=deg_K(d_n) = 2^{n-1}$, for all $n\geq 1$.
\end{cor}

It remains to obtain a similar result for the field \Q.

Similarly to what we deduced above from equations (69) and (71), we have $\Q(c_{n-1}) < \Q(d_n)$ and $\Q(d_{n-1}) < \Q(c_n)$, for all $n\geq 2$.  We now compute:

\[\Q(c_1)=\Q(\sqrt{3})= K = K(c_1),\]
and, similarly, $\Q(d_1)=K(d_1)$.  Now, assume inductively that $\Q(c_{n-1})=K(c_{n-1})$ and
$\Q(d_{n-1})=K(d_{n-1})$.  Then, we have
\[ \Q(d_n)=\Q(c_{n-1})(d_n)=K(c_{n-1})(d_n)=K(d_n)\]
and
\[ \Q(c_n)=\Q(d_{n-1})(c_n)=K(d_{n-1})(c_n)=K(c_n).\]

Therefore,

\begin{eqnarray*}
deg_{\Q}(c_n) & = & [\Q(c_n):\Q]\\
& = & [K(c_n):\Q]\\
& = & [K(a_n):\Q]\\
& = & [K(a_n):K][K:\Q]\\
&= & 2^n.
\end{eqnarray*}

Similarly for $deg_{\Q}(d_n)$.

This completes the proof of Proposition 5

\end{appendix}

\end{document}